\documentclass{article}

\usepackage[cp1252]{inputenc}
\usepackage[OT1]{fontenc}
\usepackage[english]{babel}

\usepackage{ifpdf,amssymb, amsmath, amsthm, 
amscd,  
graphicx, 
verbatim,
cite
}

\graphicspath{{figures/}}

\newcommand{\der}[2]{\frac{d \, #1}{d\, #2} }
\newcommand{\pder}[2]{\frac{\partial \, #1}{\partial \, #2} }
\newcommand{\be}[1]{\begin{equation}\label{#1}}
\newcommand{\ee}{\end{equation}}

\newcommand{\map}[3]{#1 \, : \, #2 \to #3}

\newcommand{\eq}[1]{$(\protect\ref{#1})$}
\newcommand{\restr}[2]{\left. #1 \right|_{#2}}

\renewcommand{\Vec}{\operatorname{Vec}\nolimits}
\newcommand{\Id}{\operatorname{Id}\nolimits}
\newcommand{\ad}{\operatorname{ad}\nolimits}
\newcommand{\Ad}{\operatorname{Ad}\nolimits}
\newcommand{\tr}{\operatorname{tr}\nolimits}

\def\ds{\displaystyle}

\def\R{{\mathbb R}}
\def\N{{\mathbb N}}
\def\CL{\mathcal{L}}

\def\D{\Delta}
\def\d{\delta}
\def\lam{\lambda}
\def\p{\psi}
\def\th{\theta}
\def\g{\gamma}

\def\vH{\vec H}
\def\dlam{\dot \lambda}
\def\dq{\dot q}

\def\then{\Rightarrow}

\newcommand{\const}{\operatorname{const}\nolimits}
\newcommand{\spann}{\operatorname{span}\nolimits}
\newcommand{\Lie}{\operatorname{Lie}\nolimits}

\def\sspan{\operatorname{span}}
\newcommand{\hall}{\operatorname{Hall}\nolimits}
\newcommand{\rank}{\mathrm{rank}}

\def\tcut{t_{\operatorname{cut}}}
\def\tconj{t^1_{\operatorname{conj}}}

\newtheorem{theorem}{Theorem}
\newtheorem{lemma}{Lemma}
\newtheorem{corollary}{Corollary}
\newtheorem{proposition}{Proposition}
\newtheorem{conjecture}{Conjecture}

\theoremstyle{remark}

\def\sR{sub-Rie\-man\-ni\-an }

\newcommand{\onefiglabelsizen}[4]
{
\begin{figure}[htbp]
\begin{center}
\includegraphics[height=#4cm]{#1}
\\
\parbox[t]{0.7\textwidth}{\caption{#2}\label{#3}}
\end{center}
\end{figure}
}

\newcommand{\twofiglabelsize}[8]
{
\begin{figure}[htbp]
\includegraphics[width=#7\textwidth]{#1}
\hfill
\includegraphics[width=#8\textwidth]{#4}
\\
\parbox[t]{0.4\textwidth}{\caption{#2}\label{#3}}
\hfill
\parbox[t]{0.4\textwidth}{\caption{#5}\label{#6}}
\end{figure}
}

\title{Sub-Riemannian geodesics on  the free Carnot group with the growth vector $(2,3,5,8)$%
\footnote{Work supported by 
Grant of the Russian Federation for the State Support of Researches
(Agreement  No~14.B25.31.0029).}}

\author{
Yuri Sachkov\\
Program Systems Institute\\
Russian Academy of Sciences\\
Pereslavl-Zalessky,  Russia\\
E-mail: sachkov@sys.botik.ru}

\begin{document}

\maketitle

\bigskip
\centerline{\em To Lena, for the birthday}
\bigskip

\begin{abstract}
 We consider the free nilpotent Lie algebra
$L$  with 2 generators, of step~4, and the corresponding connected simply connected Lie group $G$. We study the left-invariant sub-Riemannian structure on $G$ defined by the generators of $L$ as an orthonormal frame.

We compute two vector field models of $L$ by polynomial vector fields in~$\R^8$, and find an infinitesimal symmetry of the   \sR structure. Further,
  we compute explicitly  
the product rule in~$G$, 
the right-invariant frame on $G$, 
linear on fibers Hamiltonians corresponding to the left-invariant and right-invariant frames  on $G$,  
Casimir functions and co-adjoint orbits on  $L^*$.

Via Pontryagin maximum principle,  we describe abnormal extremals and derive a Hamiltonian system $\dlam = \vH(\lam)$, $\lam \in T^*G$, for normal extremals. 
 We compute 10 independent integrals of $\vH$, of which only 7 are in involution.  After reduction by 4 Casimir functions, the vertical subsystem of $\vH$ on $L^*$  shows numerically a chaotic dynamics, which leads to a conjecture  on non-integrability of~$\vH$ in the Liouville sense.
\end{abstract}

\section{Introduction}\label{sec:intro}
In this work we study a variational problem that can be stated equivalently in the following three ways.

\bigskip
{\em (1) Geometric statement.}
Consider two points $a_0, a_1 \in \R^2$ connected by a smooth curve $\g_0 \subset \R^2$. Fix arbitrary data $ S \in \R$, $c = (c_x, c_y) \in \R^2$, $M = (M_{xx}, M_{xy}, M_{yy}) \in \R^3$. The problem is to connect the points $a_0$, $a_1$ by the shortest smooth curve $\g\subset \R^2$ such that the domain $D \subset \R^2$ bounded by $\g_0 \cup \g$ satisfy the following properties:
\begin{enumerate}
\item
$\textrm{area}(D) = S$,
\item
$\textrm{center of mass}(D)=c$,
\item
$\textrm{second order moments}(D) = M$.
\end{enumerate}

\bigskip
{\em (2) Algebraic statement.}
Let $L$ be the free nilpotent Lie algebra with two generators $X_1$, $X_2$ of step 4:
\begin{align}
L &= \spann(X_1, \dots, X_8) \label{L}, \\
[X_1, X_2] &= X_3, \label{X12i}\\
[X_1, X_3] &= X_4, \quad  [X_2, X_3] = X_5, \label{X13i}\\
[X_1, X_4] &= X_6, \quad [X_1, X_5] = [X_2, X_4] = X_7, \quad [X_2, X_5]= X_8. \label{X14i}
\end{align}
Let $G$ be the connected simply connected Lie group with the Lie algebra $L$, we consider $X_1$, \dots, $X_8$ as a frame of left-invariant vector fields on $G$.
Consider the left-invariant sub-Riemannian structure $(G, \D, g)$ defined by $X_1$, $X_2$ as an orthonormal frame:
$$
\D_q = \spann(X_1(q), X_2(q)), \qquad g(X_i, X_j) = \d_{ij}.
$$
The problem is to find sub-Riemannian length minimizers that connect two given points $q_0, q_1 \in G$:
\begin{align*}
&q(t) \in G, \qquad q(0) = q_0, \quad q(t_1) = q_1, \\
&\dq(t) \in \D_{q(t)}, \\
&l = \int_0^{t_1} \sqrt{g(\dq, \dq)} \, dt \to \min.
\end{align*}

\bigskip
{\em (3) Optimal control  statement.}
Let vector fields $X_1, X_2 \in \Vec(\R^8)$ be defined by~\eq{X1}, \eq{X2}. Given arbitrary points $q_0, q_1 \in \R^8$, it is required to find solutions of the optimal control problem
\begin{align}
&\dot q = u_1 X_1(q) + u_2 X_2(q), \qquad q \in \R^8, \quad (u_1,u_2) \in \R^2, \label{sys}\\
&q(0) = q_0, \qquad q(t_1) = q_1, \label{bound}\\
&J = \frac 12 \int_0^{t_1} (u_1^2 + u_2^2) \, dt \to \min. \label{J}
\end{align}

\bigskip
The problem stated will be called the nilpotent sub-Riemannian problem with the growth vector $(2,3,5,8)$, or just the $(2,3,5,8)$-problem. 
There are several important motivations for the study of this problem:
\begin{itemize}
\item
this problem is a nilpotent approximation of a general sub-Riemannian  problem with the growth vector (2,3,5,8)~\cite{gromov, mitchell, bellaiche, agrachev_sarychev, mont},
\item
this problem is a natural continuation of the basic sub-Riemannian (SR) problems: the nilpotent SR problem on the Heisenberg group (aka Dido's problem, growth vector (2,3))~\cite{brock, versh_gersh}, and the nilpotent SR problem on the Cartan group (aka generalized Dido's problem, growth vector (2,3,5))~\cite{dido_exp, max1, max2, max3},
\item
this problem is included into a natural infinite chain of rank 2 SR problems with the  free nilpotent Lie algebras of step $r$, $r \in \N$, and more generally into a natural 2-dimensional lattice of rank $d$ SR problems with   the  free nilpotent Lie algebras of step $r$, $(d,r) \in \N^2$,
\item
this problem is the simplest possible SR problem on a step 4 Carnot group, and it is the first SR problem with growth vector of length 4 that should be studied.
\end{itemize}

To the best of our knowledge, this is the first study of the (2,3,5,8)-problem (although, it was mentioned in~\cite{monti} as a SR problem with smooth abnormal minimizers). 

The structure of this work is as follows. 

In Sec.~\ref{sec:realis} we construct   
two models (``asymmetric'' and ``symmetric'') of the free nilpotent Lie algebra with 2 generators of step 4 by polynomial vector fields in $\R^8$. For these models, we use respectively an algorithm due to Grayson and Grossman~\cite{grayson_grossman1} and an original approach. In the symmetric model, a one-parameter group of symmetries leaving the initial point fixed is found. 

In Sec.~\ref{sec:group} we describe explicitly the product rule in the Lie group $G \cong \R^8$, construct a right-invariant frame on $G$ corresponding naturally to the left-invariant frame given by $X_1$, $X_2$ and their iterated Lie brackets, compute the corresponding left-invariant and right-invariant Hamiltonians that are linear on fibers of $T^*G$, describe Casimir functions and co-adjoint orbits in the dual space $L^*$ of the Lie algebra $L$.

In Sec.~\ref{sec:PMP} we apply Pontryagin maximum principle to the (2,3,5,8)-problem: we describe abnormal extremals and derive a Hamiltonian system $\dlam = \vH(\lam)$, $\lam \in T^*G$, for normal extremals. 

In Sec.~\ref{sec:normal}
we study integrability of the  normal Hamiltonian field $\vH$. We compute 10 independent integrals of $\vH$, of which only 7 are in involution.  After reduction by 4 Casimir functions, the vertical subsystem of $\vH$ on $L^*$ shows numerically a chaotic dynamics, which leads to a conjecture  on non-integrability of~$\vH$.

In Sec.~\ref{sec:concl} we suggest possible questions for further study.


\section{Realisation by polynomial vector fields in $\R^8$}\label{sec:realis}
In this section we construct two models of the free nilpotent Lie algebra $L$\eq{L}--\eq{X14i} by polynomial vector fields in $\R^8$.

\subsection{Free nilpotent Lie algebras}

Let $\CL_d$ be the real free Lie algebra with $d$ generators~\cite{reutenauer}; $\CL_d$ is the Lie algebra of commutators  of $d$ variables. We have $\CL_d=\oplus^\infty_{i=1}\CL_d^i$, where $\CL^i_d$  is the space of  commutator polynomials of degree $i$. Then 
$\CL^{(r)}_d := \CL_d / \oplus^\infty_{i=r+1}\CL_d^i  $
is the free nilpotent  Lie algebra with $d$ generators of step $r$. 

Denote
$l_d(i):=\dim \CL^i_d$, $l^{(r)}_d := \dim \CL^{(r)}_d = \sum^r_{i=1}l_d(i)$.
The classical   expression of $l_d(i)$ is 
$il_d(i)=d^i - \sum_{j|i,\ 1 \leq j<i} jl_d(j)$.

In this  work we are interested  in free nilpotent Lie algebras with $2$ generators. Dimensions of such Lie algebras for small step are given in Table~1.   
\begin{table}[htbp]
\label{tab:dim}
\begin{center}
\begin{tabular}{|c|r|r|r|r|r|r|r|r|r|r|}
\hline
$i $&      1 & 2 & 3  & 4 & 5& 6& 7&   8& 9& 10\\
\hline
$l_2(i)$ & 2 & 1 & 2 & 3 &  6& 9& 18& 30& 56& 99\\
\hline
$l_2^{(i)}$ & 2 & 3 & 5 & 8 & 14&23& 41& 71&127&226\\
\hline
\end{tabular}
\end{center}
\caption{Dimensions of free nilpotent Lie algebras $\CL^{(i)}_2$}
\end{table}

\subsection{Carnot algebras and groups}

A Lie algebra $L$ is called a Carnot algebra if it admits a decomposition
$L= \oplus^r_{i=1} L_i$
as a vector space, such that
$[L_i, L_j] \subset L_{i+ j}$, 
$L_s = {0} \text{ for } s>r$,
$L_{i+1}=[L_1, L_i]$.

A free nilpotent Lie algebra $\CL^{(r)}_d$ is a Carnot algebra with the homogeneous components $L_i=\CL^i_d.$

A Carnot group $G$ is a connected, simply connected Lie group whose Lie algebra $L$ is a Carnot algebra. If $L$ is realized as the Lie algebra of left-invariant vector fields on $G$, then the degree $1$ component $L_1$ can be thought of as a completely nonholonomic (bracket-generating) distribution  on $G$. If moreover $L_1$ is endowed with a left-invariant inner product
$g$, then ($G, L_1, g $)  becomes a nilpotent left-invariant \sR manifold~\cite{mont}. Such \sR  structures  are nilpotent approximations of generic \sR  structures~\cite{gromov, mitchell, bellaiche, agrachev_sarychev}.

The sequence of numbers
$$ \left( \dim L_1, \dim L_1+\dim L_2, \dots, \dim L_1 + \dots + \dim L_r = \dim L \right)$$
is called the growth vector of the distribution $ L_1$~\cite{versh_gersh}.

For free nilpotent Lie algebras, the growth vector is maximal compared with all Carnot algebras with the bidimension $(\dim L_1, \dim L)$.

\subsection{Lie algebra with the growth vector $(2, 3, 5, 8)$}\label{subsec:algebra}

The Carnot algebra with the growth vector (2, 3, 5, 8) 
$$\CL^{(4)}_2 = \sspan (X_1, \dots, X_8)$$
is determined by the following  multiplication table:
\begin{align}
[X_1, X_2] &= X_3, \label{X1X2}\\
[X_1, X_3] &= X_4, \quad  [X_2, X_3] = X_5, \label{X1X3}\\
[X_1, X_4] &= X_6, \quad [X_1, X_5] = [X_2, X_4] = X_7, \quad [X_2, X_5]= X_8, \label{X1X5}
\end{align}
 with all the rest brackets equal to zero. 
 This multiplication table is depicted at Fig.~\ref{fig:2358}.

\begin{figure}[htb]
\setlength{\unitlength}{1cm}

\begin{center}
\begin{picture}(4, 4)
\put(1.15, 3.9){ \vector(1, -1){0.8}}
\put(1, 3.9){ \vector(0, -1){2}}
\put(2.85, 3.9){ \vector(-1, -1){0.8}}
\put(3, 3.9){ \vector(0, -1){2}}
\put(1.9, 2.65){ \vector(-1, -1){0.8}}
\put(2.1, 2.65){ \vector(1, -1){0.8}}
\put(0.95, 3.9){ \vector(-1, -4){0.8}}
\put(1, 1.45){ \vector(-1, -1){0.8}}
\put(3.1, 3.9){ \vector(1, -4){0.8}}
\put(3.05, 1.45){ \vector(1, -1){0.8}}
\put(1.05, 1.45){ \vector(1, -1){0.8}}
\put(2.9, 3.9){ \vector(-1, -4){0.8}}
\thicklines
\put(2.95, 1.45){ \vector(-1, -1){0.8}}
\put(1.1, 3.9){ \vector(1, -4){0.8}}


\put(1, 1.5) {$X_4$}
\put(3, 1.5) {$X_5$}
\put(1, 3.98) {$X_1$}
\put(3, 3.98) {$X_2$}
\put(2, 2.75) {$X_3$}
\put(0, 0.25) {$X_6$}
\put(2, 0.25) {$X_7$}
\put(4, 0.25) {$X_8$}

\end{picture}

\caption{Lie algebra with the growth vector $(2, 3, 5, 8)$\label{fig:2358}}
\end{center}
\end{figure}
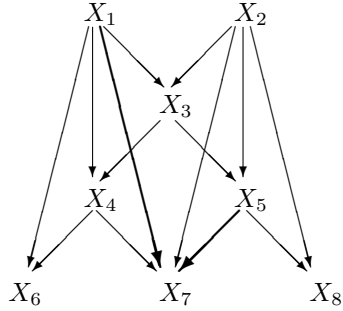


\subsection{Hall basis}

Free nilpotent Lie algebras have a convenient basis introduced by M. Hall~\cite{hall}.
We describe it using the exposition of~\cite{grayson_grossman1}.

The Hall basis of the free Lie algebra $\CL_d$ with $d$ generators $X_1$, \dots, $X_d$
is the subset $\hall \subset \CL_d$ that has a decomposition into homogeneous components
$
\hall = \cup_{i=1}^{\infty} \hall_i
$
defined as follows.

Each element $H_j$, $j = 1, 2, \dots$, of the Hall basis  is a monomial in the generators $X_i$ and is defined recursively as follows. The generators satisfy the inclusion
$
X_i \in \hall_1$, $i = 1, \dots, d$,
and we denote
$
H_i = X_i$, 
$i = 1, \dots, d$.
If we have defined basis elements 
$
H_1, \dots, H_{N_{p-1}} \in \oplus_{j=1}^{p-1} \hall_j$,
they are simply ordered so that $E < F$ if $E \in \hall_k$, $F \in \hall_l$, $k < l$:
$
H_1 < H_2 < \dots < H_{N_{p-1}}$.
Also if $E \in \hall_s$, $F \in \hall_t$ and $p = s +t$, then
$
[E,F] \in \hall_p
$
if:
\begin{enumerate}
\item
$E > F$, and
\item
if $E = [G,K]$, then $K \in \hall_q$ and $t \geq q$.
\end{enumerate}

By this definition, one easily computes recursively the first components $\hall_i$ 
of the Hall basis
for $d = 2$:
\begin{align*}
&\hall_1 = \{H_1, H_2\}, \qquad H_1 = X_1, \quad H_2 = X_2, \\
&\hall_2 = \{H_3\}, \qquad H_3 = [X_2, X_1], \\
&\hall_3 = \{H_4, H_5\}, \qquad H_4 = [[X_2, X_1],X_1], \quad  H_5 = [[X_2, X_1],X_2], \\
&\hall_4 = \{H_6, H_7, H_8\}, \\ 
&H_6 = [[[X_2, X_1],X_1],X_1], \  H_7 = [[[X_2, X_1],X_1],X_2], \  H_8 = [[[X_2, X_1],X_2],X_2].
\end{align*}
Consequently,
$
\CL_2^{(4)} = \spann\{H_1, \dots, H_8\}$.
In the sequel we use a more convenient basis of
$
\CL_2^{(4)} = \spann\{X_1, \dots, X_8\}
$
with the multiplication table~\eq{X1X2}--\eq{X1X5}.

\subsection{Asymmetric vector field model for $\CL_2^{(4)}$}
Here we recall an algorithm for construction of a vector field model for the Lie algebra $\CL_2^{(r)}$ due to Grayson and Grossman~\cite{grayson_grossman1}.
For a given $r \geq 1$, the algorithm evaluates two polynomial vector fields $H_1, H_2 \in \Vec(\R^N)$, $N = \dim \CL_2^{(r)}$, which generate the Lie algebra $\CL_2^{(r)}$. 

Consider the Hall basis elements
$
\spann\{H_1, \dots, H_N\} = \CL_2^{(r)}$.
Each element $H_i \in \hall_j$ is a Lie bracket of length $j$:
\begin{align*}
&H_i = [\dots[[H_2, H_{k_j}],H_{k_{j-1}}], \dots, H_{k_1}],\\
&k_j = 1, \qquad k_{n+1} \leq k_n \text{ for } 1 \leq n \leq j-1.
\end{align*} 
This defines a partial ordering of the basis elements. We say that $H_i$ is a direct descendant of $H_2$ and of each $H_{k_l}$ and write $i \succ 2$, $i \succ k_l$, $l = 1, \dots, j$.

Define monomials $P_{2,k}$ in $x_1$, \dots, $x_N$ inductively by 
$$
P_{2,k} = - x_j \ P_{2,i} /(\deg_j P_{2,i} + 1),
$$
  whenever $H_k = [H_i,H_j]$ is a basis Hall element, and where  $\deg_j P$ is the highest power of $x_j$ which divides $P$. 

The following theorem gives the properties of the generators.

\begin{theorem}[Th. 3.1 \cite{grayson_grossman1}]\label{th:Hall}
Let $r \geq 1$ and let $N = \dim  \CL_2^{(r)}$. Then the vector fields
$
H_1 = \ds\pder{}{x_1}$, $H_2 = \ds\pder{}{x_2} + \sum_{i \succ 2} P_{2,i} \pder{}{x_i}
$
have the following properties:
\begin{enumerate}
\item
they are homogeneous of weight one with respect to the grading
$$
\R^N = \hall_1 \oplus \dots \oplus \hall_r;
$$
\item
$\Lie(H_1, H_2) = \CL_2^{(r)}$.
\end{enumerate}
\end{theorem}

The algorithm described before Theorem~\ref{th:Hall} produces the following vector
field basis of $\CL_2^{(4)}$:
\begin{align*}
    H_1 &= \frac{\partial }{\partial x_1},\\
    H_2 &= \frac{\partial }{\partial x_2} - x_1\frac{\partial }{\partial x_3} - \frac{x_1^2}{2}\frac{\partial }{\partial x_4}
         - x_1x_2\frac{\partial }{\partial x_5} +  \frac{x_1^3}{6}\frac{\partial }{\partial x_6} + \frac{x_1^2x_2}{2}\frac{\partial }{\partial x_7}
         + \frac{x_1x_2^2}{2}\frac{\partial }{\partial x_8},\\
    H_3 &= \frac{\partial }{\partial x_3} + x_1\frac{\partial }{\partial x_4} + x_2\frac{\partial }{\partial x_5}
         - \frac{x_1^2}{2}\frac{\partial }{\partial x_6} - x_1x_2\frac{\partial }{\partial x_7}
         - \frac{x_2^2}{2}\frac{\partial }{\partial x_8},\\
    H_4 &= -\frac{\partial }{\partial x_4} + x_1\frac{\partial }{\partial x_6} + x_2\frac{\partial }{\partial x_7},\\
    H_5 &= -\frac{\partial }{\partial x_5} + x_1\frac{\partial }{\partial x_7} + x_2\frac{\partial }{\partial x_8},\\
    H_6 &= -\frac{\partial }{\partial x_6},\\
    H_7 &= -\frac{\partial }{\partial x_7},\\
    H_8 &= -\frac{\partial }{\partial x_8},
\end{align*}
with the multiplication table
\begin{align}
    \left[ H_2, H_1 \right] &= H_3, \label{H2H1}\\
    \left[ H_3, H_1 \right] &= H_4, \; \left[ H_3, H_2 \right] = H_5, \label{H3H1} \\
    \left[ H_4, H_1 \right] &= H_6, \; \left[ H_4, H_2 \right] = H_7, \; 
    \left[ H_5, H_2 \right] = H_8. \label{H4H1}
\end{align}

\subsection{Symmetric vector field model of $\CL_2^{(4)}$}\label{subsec:symmetr}
The vector field model of the Lie algebra $\CL_2^{(4)}$ via the fields $H_1,\dotsc, H_8$
obtained in the previous subsection is asymmetric in the sense that
there is no visible symmetry between the vector fields $H_1$ and $H_2$.
Moreover, no continuous symmetries of the \sR structure
generated by the orthonormal frame $\left\{ H_1, H_2 \right\}$ are visible,
although the Lie brackets \eqref{H2H1}--\eqref{H4H1} suggest that this 
sub-Riemannian structure should be preserved by a one-parameter group of
rotations in the plane $\sspan\{H_1, H_2\}$.

One can find a symmetric vector field model of $\CL_2^{(4)}$ free of such
shortages as in the following statement.

\begin{theorem} \label{th:Xi}
    \begin{itemize}
        \item[$(1)$] 
            The vector fields
\begin{align}
&X_1 = \pder{}{x_1} - \frac{x_2}{2} \pder{}{x_3} - \frac{x_1^2 + x_2^2}{2} \pder{}{x_5} - \frac{x_1 x_2^2}{4} \pder{}{x_7} - \frac{x_2^3}{6} \pder{}{x_8}, \label{X1} \\
&X_2 = \pder{}{x_2} + \frac{x_1}{2} \pder{}{x_3} + \frac{x_1^2 + x_2^2}{2} \pder{}{x_4} + \frac{x_1^3}{6} \pder{}{x_6} + \frac{x_1^2 x_2}{4} \pder{}{x_7}, \label{X2}\\
&X_3 = \pder{}{x_3} + x_1 \pder{}{x_4} + x_2 \pder{}{x_5} + \frac{x_1^2}{2} \pder{}{x_6} + x_1 x_2 \pder{}{x_7} + \frac{x_2^2}{2} \pder{}{x_8}, \label{X3}\\
&X_4 = \pder{}{x_4} + x_1 \pder{}{x_6} + x_2 \pder{}{x_7}, \label{X4}\\
&X_5 = \pder{}{x_5} + x_1 \pder{}{x_7} + x_2 \pder{}{x_8}, \label{X5}\\
&X_6 = \pder{}{x_6}, \label{X6}\\
&X_7 = \pder{}{x_7}, \label{X7}\\
&X_8 = \pder{}{x_8} \label{X8}
\end{align}         
            satisfy the multiplication table~\eq{X1X2}--\eq{X1X5}.
            Thus the fields $X_1,\dotsc, X_8 \in \Vec(\R^8)$ model the Lie
            algebra $\CL_2^{(4)}$.
        \item[$(2)$] 
            The vector field 
\begin{align}
&X_0 = x_2 \pder{}{x_1} - x_1 \pder{}{x_2} + x_5 \pder{}{x_4} - x_4 \pder{}{x_5} + P \pder{}{x_6} + Q \pder{}{x_7} + R \pder{}{x_8},  \label{X0}\\
&P = - \frac{x_1^4}{24} +   \frac{x_1^2 x_2^2}{8} + x_7, \label{P}\\ 
&Q = \frac{x_1 x_2^3}{12} +   \frac{x_1^3 x_2}{12} - 2 x_6 + 2 x_8, \label{Q}\\
&R = \frac{x_1^2 x_2^2}{8} - \frac{x_2^4}{24} -    x_7 \label{R}
\end{align}
            satisfies the following relations:
\begin{align}
&[X_0, X_1] = X_2, \qquad   [X_0, X_2] = - X_1, \qquad [X_0, X_3] = 0, \label{X0X1}\\
&[X_0, X_4] = X_5, \qquad   [X_0, X_5] = - X_4, \label{X0X4}\\
&[X_0, X_6] = 2 X_7, \qquad   [X_0, X_7] = X_8 - X_6, \qquad   [X_0, X_8] = - 2 X_7. \label{X0X6}
\end{align}
            Thus the field $X_0$ is an infinitesimal symmetry of the sub-Riemannian structure
            generated by the orthonormal frame $\left\{ X_1, X_2 \right\}$. 
    \end{itemize}
\end{theorem}

\begin{proof}

    In fact, the both statements of the proposition are verified by the direct
computation, but we prefer to describe a method of construction of the
vector fields $X_1,\dotsc, X_8$, and $X_0$.

$(1)$
In the previous work \cite{dido_exp} we constructed a similar symmetric
vector field model for the Lie algebra $\CL_2^{(3)}$, which has growth vector (2, 3, 5):
\begin{align}
    \CL_2^{(3)} &= \sspan\{X_1,\dotsc, X_5\} \subset \Vec(\R^5), \\
     X_1 &=  \pder{}{x_1} - \frac{x_2}{2} \pder{}{x_3} - \frac{x_1^2 + x_2^2}{2} \pder{}{x_5}, \label{X1L23}\\
     X_2 &= \pder{}{x_2} + \frac{x_1}{2} \pder{}{x_3} + \frac{x_1^2 + x_2^2}{2} \pder{}{x_4},\label{X2L23}\\
     X_3 &= \pder{}{x_3} + x_1 \pder{}{x_4} + x_2 \pder{}{x_5},\label{X3L23}\\
     X_4 &= \pder{}{x_4}, \label{X4L23}\\
     X_5 &= \pder{}{x_5}, \label{X5L23}
\end{align}
with the Lie brackets~\eq{X1X2}, \eq{X1X3}. 
Now we aim to ``continue'' these relationships to vector fields
$X_1,\dotsc, X_8 \in \Vec(\R^8)$ that span the Lie algebra $\CL_2^{(4)}$.
So we seek for vector fields of the form
\begin{align}
    X_1 &= \frac{\partial }{\partial x_1} - \frac{x_2}{2}\frac{\partial }{\partial x_3} - \frac{x_1^2 + x_2^2}{2}\frac{\partial }{\partial x_5}
    + \sum\limits_{i=6}^8 a_1^i\frac{\partial }{\partial x_i},  \; \label{X1aij}   \\
    X_2 &= \frac{\partial }{\partial x_2} + \frac{x_1}{2}\frac{\partial }{\partial x_3} - \frac{x_1^2 + x_2^2}{2}\frac{\partial }{\partial x_4}
    + \sum\limits_{i=6}^8 a_2^i\frac{\partial }{\partial x_i}, \;  \label{X2aij} \\
    X_3 &= \frac{\partial }{\partial x_3} + x_1\frac{\partial }{\partial x_4} + x_2\frac{\partial }{\partial x_5}
    + \sum\limits_{i=6}^8 a_3^i\frac{\partial }{\partial x_i}, \; \label{X3aij} \\ 
    X_4 &= \frac{\partial }{\partial x_4} + \sum\limits_{i=6}^8 a_4^i\frac{\partial }{\partial x_i}, \; \label{X4aij} \\ 
    X_5 &= \frac{\partial }{\partial x_5} + \sum\limits_{i=6}^8 a_5^i\frac{\partial }{\partial x_i}, \; \label{X5aij}\\ 
    X_j &= \sum\limits_{i=6}^8 a_i^j\frac{\partial }{\partial x_j}, \qquad j = 6, 7, 8, 
\end {align}
such that $\sspan\{X_1,\dotsc, X_8\} = \CL_2^{(4)}$.

Compute the required Lie brackets:

\begin{align*}
    \left[ X_1, X_2 \right] &= \frac{\partial }{\partial x_3} + x_1\frac{\partial }{\partial x_4} + x_2\frac{\partial }{\partial x_5}
    + \left( \frac{\partial a_2^6}{\partial x_1} - \frac{\partial a_1^6}{\partial x_2} \right)\frac{\partial }{\partial x_6}\\
    &+ \left( \frac{\partial a_2^7}{\partial x_1} - \frac{\partial a_1^7}{\partial x_2} \right)\frac{\partial }{\partial x_7}
    + \left( \frac{\partial a_2^8}{\partial x_1} - \frac{\partial a_1^8}{\partial x_2} \right)\frac{\partial }{\partial x_8},\\
    \left[ X_1, X_3 \right] &= \frac{\partial }{\partial x_4} + \frac{\partial a_3^6}{\partial x_1}\frac{\partial }{\partial x_6}
    + \frac{\partial a_3^7}{\partial x_1}\frac{\partial }{\partial x_7} + \frac{\partial a_3^8}{\partial x_1}\frac{\partial }{\partial x_8},\\
    \left[ X_2, X_3 \right] &= \frac{\partial }{\partial x_5} + \frac{\partial a_3^6}{\partial x_2}\frac{\partial }{\partial x_6}
    + \frac{\partial a_3^7}{\partial x_2}\frac{\partial }{\partial x_7} + \frac{\partial a_3^8}{\partial x_2}\frac{\partial }{\partial x_8},\\
    \left[ X_1, X_4 \right] &= \frac{\partial a_4^6}{\partial x_1}\frac{\partial }{\partial x_6} + \frac{\partial a_4^7}{\partial x_1}\frac{\partial }{\partial x_7}
    + \frac{\partial a_4^8}{\partial x_1}\frac{\partial }{\partial x_8},\\
    \left[ X_1, X_5 \right] &= \frac{\partial a_5^6}{\partial x_1}\frac{\partial }{\partial x_6} + \frac{\partial a_5^7}{\partial x_1}\frac{\partial }{\partial x_7}
    + \frac{\partial a_5^8}{\partial x_1}\frac{\partial }{\partial x_8},\\
    \left[ X_2, X_4 \right] &= \frac{\partial a_4^6}{\partial x_2}\frac{\partial }{\partial x_6} + \frac{\partial a_4^7}{\partial x_2}\frac{\partial }{\partial x_7}
    + \frac{\partial a_4^8}{\partial x_2}\frac{\partial }{\partial x_8},\\
    \left[ X_2, X_5 \right] &= \frac{\partial a_5^6}{\partial x_2}\frac{\partial }{\partial x_6} + \frac{\partial a_5^7}{\partial x_2}\frac{\partial }{\partial x_7}
    + \frac{\partial a_5^8}{\partial x_2}\frac{\partial }{\partial x_8}.
\end{align*}

The vector fields $X_1,\dotsc, X_8$ should be independent, thus the determinant constructed of these vectors as columns should satisfy the inequality
\begin{align*}
D = \det \left(X_1,\dotsc, X_8\right) = 
\begin{vmatrix}
a_6^6&a_7^6&a_8^6\\
a_6^7&a_7^7&a_8^7\\
a_6^8&a_7^8&a_8^8
\end{vmatrix} \neq 0.
\end{align*}
We will choose $a_i^j$ such that $D = 1$. It follows from the multiplication table
for $X_1,\dotsc, X_8$ that
\begin{align*}
D =  
\begin{vmatrix}
\ds    \frac{d^2a_3^6}{dx_1^2} & \ds\frac{d^2a_3^6}{dx_1dx_2}  & \ds\frac{d^2a_3^6}{dx_2^2}\\
\ds    \frac{d^2a_3^7}{dx_1^2} & \ds\frac{d^2a_3^7}{dx_1dx_2}  & \ds\frac{d^2a_3^7}{dx_2^2}\\
 \ds   \frac{d^2a_3^8}{dx_1^2} & \ds\frac{d^2a_3^8}{dx_1dx_2}  & \ds\frac{d^2a_3^8}{dx_2^2} 
\end{vmatrix}.
\end{align*}
In order to get $D = 1$, define the entries of this matrix in the following symmetric way:
$a_3^6 = \ds\frac{x_1^2}{2}$, $a_3^7 = x_1x_2$, $a_3^8 = \ds\frac{x_2^2}{2}$.  
Then we obtain from the multiplication table for $X_1,\dotsc, X_8$ that
$\ds\frac{\partial a_2^6}{\partial x_1} - \frac{\partial a_1^6}{\partial x_2} = a_3^6 = \frac{x_1^2}{2}$,  
$\ds    \frac{\partial a_2^7}{\partial x_1} - \frac{\partial a_1^7}{\partial x_2} = a_3^7 = x_1x_2$,
    $\ds\frac{\partial a_2^8}{\partial x_1} - \frac{\partial a_1^8}{\partial x_2} = a_3^8 = \frac{x_2^2}{2}$.  
We solve these equations in the following symmetric way:
    $a_1^6 = 0$, $a_2^6 = \ds\frac{x_1^3}{6}$, 
    $a_1^7 = -\ds\frac{x_1x_2^2}{4}$, $a_2^7 = \ds\frac{x_1^2x_2}{4}$,
    $a_1^8 = -\ds\frac{x_2^3}{6}$, $a_2^8 = 0$.  
Then we substitute these coefficients to~\eq{X1aij}, \eq{X2aij} 
and check item (1) of this theorem by direct computation.

Now we prove item $(2)$.
We proceed exactly as for item $(1)$: we start from an infinitesimal symmetry~\cite{dido_exp}
\be{X0L23}
X_0 = x_2 \pder{}{x_1} - x_1 \pder{}{x_2} + x_5 \pder{}{x_4} - x_4 \pder{}{x_5} \in \Vec(\R^5)
\ee
of the \sR structure on $\R^5$ determined by the orthonormal frame \eq{X1L23}, \eq{X2L23} and ``continue'' symmetry~\eq{X0L23} to the \sR structure on $\R^8$ determined by the orthonormal frame~\eq{X1}, \eq{X2}.

So we seek for a vector field $X_0 \in \Vec(\R^8)$ of the form~\eq{X0} for the functions $P, Q, R \in C^{\infty}(\R^8)$ to be determined so that the multiplication table~\eq{X0X1}--\eq{X0X6} hold.

The first two equalities in~\eq{X0X1} yield
$
X_1 P = - \ds\frac{x_1^3}{6}$, $\ds X_2 P = \frac{x_1^2 x_2}{2}$.
Further,
$X_3 P = [X_1, X_2] P = X_1 X_2 P - X_2 X_1 P = X_1 \ds\frac{x_1^2 x_2}{2} + X_2 \frac{x_1^3}{6} = x_1 x_2$.
Similarly it follows that
$
X_4 P = x_2$, 
$X_5 P = x_1$, 
$X_6 P = 0$,   
$X_7 P = 1$, 
$X_8 P = 0$. 
Since $X_6 P = X_8 P = 0$, then $P = P(x_1, x_2, x_3, x_4, x_5, x_7)$. Moreover, since $X_7 P = 1$, then $P =  x_7 + a(x_1, x_2, x_3, x_4, x_5)$.
The equality
$X_5 P = x_1$ implies that $\pder{a}{x_5} = 0$, i.e., $a =   a(x_1, x_2, x_3, x_4)$.
Similarly, 
since
$X_4 P = x_2$, then  $a =   a(x_1, x_2, x_3)$.
It follows from the equality
$X_3 P = x_1x_2$ that $\ds\pder{a}{x_3} = x_1x_2$, i.e.,   $a = x_1x_2x_3 +   b(x_1, x_2)$.
Moreover, the equality
$X_2 P = \ds\frac{x_1^2x_2}{2}$ implies that $\ds\pder{b}{x_2} = -x_1x_3 - \frac{x_1^2x_2}{4}$, i.e.,   $b = -x_1x_2x_3 -\ds\frac{x_1^2x_2^2}{8} +   c(x_1)$.
Finally, the equality
$X_1 P = -\ds\frac{x_1^3}{2}$ implies that $\ds\der{c}{x_1} = -\frac{x_1^3}{6} + \frac{x_1x_2^2}{2}$
i.e.,
$c = -\ds\frac{x_1^4}{24} + \frac{x_1^2x_2^2}{4}$.
Thus equality~\eq{P} follows. Similarly we get equalities~\eq{Q}, \eq{R}.

Then multiplication table~\eq{X0X1}--\eq{X0X6} for the vector field~\eq{X0}--\eq{R} is verified by a direct computation.  
\end{proof}

\section{Carnot group}\label{sec:group}
In this section we study the Carnot group $G$ with the Lie algebra $L = \CL_2^{(4)}$.

\subsection{Product rule in $G$}\label{subsec:product}
In this subsection we compute the product rule 
in the connected simply connected Lie group $G$ with the Lie algebra $L = \CL_2^{(4)}$ 
on which the vector fields $X_1,\dotsc, X_8$ given by \eqref{X1}--\eqref{X8}
are left-invariant.

Our algorithm for computation of the product rule in a Lie group $G$ with a known
left-invariant frame  $X_1,\dotsc, X_n\in \Vec(G)$ follows from the next
argument.
Let $g_1, g_2 \in G$, and let 
$g_2 = e^{t_n X_n} \circ\ldots\circ e^{t_1 X_1}(\Id)$, $t_1,\ldots,t_n \in \R$, 
where we denote by $\map{e^{tX}}{G}{G}$ the flow of the vector field $X$.
Then
    $g_1 \cdot g_2 = g_1 \cdot e^{t_n X_n} \circ\ldots\circ e^{t_1 X_1}(\Id) 
           = e^{t_n X_n} \circ\ldots\circ e^{t_1 X_1}(g_1)$
by left-invariance of $X_i$. 
So an algorithm for computation of $g_1 \cdot g_2$ is the  following:
\begin{enumerate}
    \item Compute $e^{t_i X_i}(g)$, \quad $t_i \in \R, \; g \in G$.
    \item Compute $e^{t_n X_n} \circ\ldots\circ e^{t_1 X_1}(g), \quad t_i \in \R, \; g \in G$.
    \item Solve the equation $e^{t_n X_n} \circ\ldots\circ e^{t_1 X_1}(\Id) = g_2$
        for $t_1,\ldots,t_n \in \R$ (we assume that this is possible in a unique way).
    \item Compute $g_1 \cdot g_2 = e^{t_n X_n} \circ\ldots\circ e^{t_1 X_1}(g_2)$.
\end{enumerate}
By this algorithm, we compute the product $z = x \cdot y$ in the coordinates on $G$ (notice that  as a manifold $G=\R^8$), as follows:
\begin{align*}
    x &= (x_1,\dotsc, x_8),\ y = (y_1,\dotsc, y_8), \ z = (z_1,\dotsc, z_8) \in G = \R^8,\\
    z_1 &= x_1 + y_1,\\
    z_2 &= x_2 + y_2,\\
    z_3 &= x_3 + y_3 + \frac{1}{2}(x_1y_2 - x_2y_1),\\
    z_4 &= x_4 + y_4 + \frac{1}{2}(x_1(x_1 + y_1) + x_2(x_2 + y_2) + x_1y_3),\\
    z_5 &= x_5 + y_5 - \frac{1}{2}y_1(x_1(x_1 + y_1) + x_2(x_2 + y_2)) + x_2y_3,\\ 
    z_6 &= x_6 + y_6 + \frac{x_1}{12}(2x_1^2y_2 + 3x_1y_1y_2 - 2 y_2^3 + 6x_1y_3 + 12y_4),\\ 
    z_7 &= x_7 + y_7 + \frac{1}{24}(3x_1^2y_2(2x_2 + y_2) - x_2(3x_2y_1^2 + 6y_1^2y_2 + 4(y_2^3 - 6y^4))\\
        &\qquad + x_1(-6x_2^2y_1 + 4y_1^3 + 6y_1y_2^2 + 24x_2y_3 + 24y_5)),\\
    z_8 &= x_8 + y_8 + \frac{x_2}{2}(-2x_2^2y_1 + 2y_1^3 -3x_2y_1y_2 + 6x_2y_3 + 12y_5).
\end{align*}


\subsection{Right-invariant frame on $G$}\label{subsec:right}
Computation of the right-invariant frame on $G$ corresponding
to a left-invariant frame can be done via the following simple lemma.
Denote the inversion on a Lie group $G$ as 
$i\ :\ G \rightarrow G$, $i\left(g\right) = g^{-1}$.

\begin{lemma}
    Let $X_1, X_2, X_3 \in \Vec(G)$ and $Y_1, Y_2, Y_3 \in \Vec(G)$ 
    be respectively left-invariant and right-invariant vector fields on a Lie group $G$
    such that
   $ Y_j(\Id) = -X_j(\Id)$, $j=1, 2, 3$.
    Then
\begin{align}
    i_*X_j &= Y_j, \quad i = 1, 2, 3, \label{iXiYi} \\
[X_1, X_2] &= X_3 \quad \Leftrightarrow\quad [Y_1, Y_2] = Y_3. \label{X1X2X3} 
\end{align}
\end{lemma}

\begin{proof}
Equality~\eq{iXiYi} follows by the left-invariance and right-invariance of the fields $X_i$ and $Y_i$ respectively.
Equality~\eq{X1X2X3}  follows since the diffeomorphism $i:G \rightarrow G$ preserves Lie bracket
of vector fields (see e.g. \cite{notes}).
\end{proof}

Thus if $X_1,\dots, X_n \in  \Vec G $ is a left-invariant frame on a Lie group $G$, then
$Y_1, \dots, Y_n \in \Vec G$, $Y_j=i_*X_j$,
is the right-invariant frame such that 
$Y_j(\Id) = -X_j(\Id)$, $j=1, \dots, n$,
and the same product rules as for  $X_1$, \dots, $X_n$. 

Immediate computation using the product rule in $G$ given in Subsec.~\ref{subsec:product} gives the following right-invariant frame on the Lie group $G=\R^8:$ 
\begin{align*} 
Y_1&= - \pder{}{x_1}  - \frac{x_2}{2}\pder{}{x_3 } - \frac{x_1x_2+2x_3}{2}\pder{}{ x_4} + \frac{x^2_1}{2}\pder{}{ x_5} \\
&\qquad + \frac{x^3_2-6x_4}{6}\pder{}{ x_6} - \frac{2x^3_1+3x_1x^2_2+12x_5}{12}\pder{}{ x_7},\\
Y_2&= - \pder{}{ x_2} - \frac{x_1}{2}\pder{}{ x_3} - \frac{x^2_2}{2}\pder{}{ x_4} + \frac{x_1x_2-2x_3}{2}\pder{}{ x_5}\\
&\qquad +  \frac{3x^2_1x_2+2x^3_2-12x_4}{12}\pder{}{ x_6} - \frac{x^3_1+6x_5}{6}\pder{}{ x_8},\\
Y_i&= -\pder{}{ x_i}, \qquad i=3, \dots, 8.
\end{align*}

\subsection{Left-invariant and right-invariant \\Hamiltonians on $T^*G$}\label{subsec:Ham}

Using the expressions for the left-invariant and right-invariant frames given in Subsec.~\ref{subsec:symmetr}
 and Subsec.~\ref{subsec:right}, we define the corresponding left-invariant and right-invariant Hamiltonians, linear on fibers in $T^* G$:
$$
h_i(\lambda) = \left\langle \lambda, X_i\right\rangle, \qquad 
g_i(\lambda) = \left\langle \lambda, Y_i\right\rangle
\qquad \lambda \in T^*G, \quad i=1, \dots, 8.
$$

In the  canonical coordinates
$(x_1, \dots, x_8, \psi_1, \dots, \psi_8)$ on $T^*G$~\cite{notes}
we have the following:
\begin{align*}
h_1&= \psi_1- \frac{x_2}{2}\psi_3-\frac{x^2_1+x^2_2}{2}\psi_5-\frac{x_1x^2_2}{4}\psi_7 - \frac{x^3_2}{6}\psi_8,\\
h_2&= \psi_2+ \frac{x_1}{2}\psi_3+\frac{x^2_1+x^2_2}{2}\psi_4+\frac{x^3_1}{6}\psi_6 + \frac{x^2_1x_2}{4}\psi_7,\\
h_3&= \psi_3+ x_1\psi_4+x_2\psi_5+\frac{x^2_1}{2}\psi_6 + x_1x_2\psi_7+\frac{x^2_2}{2}\psi_8,\\
h_4&= \psi_4 + x_1\psi_6+ x_2\psi_7,\\
h_5&= \psi_5 + x_1\psi_7+ x_2\psi_8,\\
h_i&=\psi_i, \qquad i=6, 7, 8,
\end{align*}
and
\begin{align}
g_1 &= -\psi_1 - \frac{x_2}{2}\psi_3 - \frac{x_1x_2+2x_3}{2}\psi_4+ \frac{x^2_1}{2}\psi_5 
\nonumber\\
&\qquad +  \frac{x^3_2-6x_4}{6}\psi_6 - \frac{2 x^3_1+3x_1x^2+12x_5}{12}\psi_7,
\label{g1}\\
g_2 &= -\psi_2 - \frac{x_1}{2}\psi_3 - \frac{x^2_2}{2}\psi_4+ \frac{x_1x_2 - 2x_3}{2}\psi_5 
\nonumber\\
&\qquad +  \frac{3x^2_1x_2+ 2x^3_2-12x_4}{12}\psi_6 - \frac{x^3_1+6x_5}{6}\psi_8,
\label{g2}\\
g_i &= -\psi_i, \qquad i=3, \dots, 8.
\label{gi}
\end{align}



\subsection{Casimir functions on  $L^*$}\label{subsec:Casimir}
In this subsection we compute Casimir functions on  the dual space $L^*$ to the Lie algebra $L = \CL_2^{(4)}$, i.e., the smooth functions 
$$\map{f}{L^*}{\R} \text{ such that }
\{f, h_i\} = 0, \qquad i = 1, \dots, 8.
$$
Simultaneously  we characterize orbits of the co-adjoint action of the Lie group $G$ on $L^*$
\be{co-orbit}
\{ \Ad_{q^{-1}}^*(h) \mid q \in G \}.
\ee

\begin{theorem}\label{th:Casimir}
The functions 
\be{Casimirs}
h_6, \quad h_7, \quad h_8, \quad 
C = 
h_5^2 h_6 - 2 h_4 h_5 h_7 + h_4^2 h_8 - 2 h_3 (h_6 h_8 -h_7^2)
\ee
are Casimir functions on $L^*$, $L = \CL_2^{(4)}$. 

If $h_6h_8-h^2_7 \neq 0,$ then these functions are independent, and any Casimir function depends functionally of them.
\end{theorem}
\begin{proof}
For all $i= 6, 7, 8,$ $j=1, \dots, 8,$ we have $[X_i, X_j] = 0$,
thus $\{h_i, h_j \}=0$. The equality $\{C, h_j\}=0$ for $j=1, \dots, 8$ is  verified immediately. Thus  $h_6$, $h_7$, $h_8$, $C$ are Casimir functions. 
Now we prove that there are no other Casimir functions on $L^*$.

Let $f \in C^\infty\left(L^*\right)$ be a Casimir function, then
\begin{align}
\left\lbrace f, h_1\right\rbrace   &= -h_3\frac{\partial  f}{\partial  h_2}-h_4\frac{\partial  f}{\partial  h_3}-h_6\frac{\partial  f}{\partial  h_4}-h_7\frac{\partial  f}{\partial  h_5}=0,\label{fh1}\\
\left\lbrace f, h_2\right\rbrace   &= h_3\frac{\partial  f}{\partial  h_1}-h_5\frac{\partial  f}{\partial  h_3}-h_7\frac{\partial  f}{\partial  h_4}-h_8\frac{\partial  f}{\partial  h_5}=0,
\label{fh2}\\
\left\lbrace f, h_3\right\rbrace   &= h_4\frac{\partial  f}{\partial  h_1}+h_5\frac{\partial  f}{\partial  h_2}=0,
\label{fh3}\\
\left\lbrace f, h_4\right\rbrace   &= h_6\frac{\partial  f}{\partial  h_1}+h_7\frac{\partial  f}{\partial  h_2}=0,
\label{fh4}\\
\left\lbrace f, h_5\right\rbrace   &= h_7\frac{\partial  f}{\partial  h_1}+h_8\frac{\partial  f}{\partial  h_2}=0.
\label{fh5}
\end{align}
These equalities are  conveniently rewritten in terms of the following vector fields  $V_i \in \Vec L^*$:
\begin{align}
 V_1&=-h_3\frac{\partial }{\partial  h_2}-h_4\frac{\partial }{\partial  h_3}-h_6\frac{\partial }{\partial  h_4}-h_7\frac{\partial }{\partial  h_5}, 
 \label{V1}\\
 V_2 &= h_3\frac{\partial }{\partial  h_1}-h_5\frac{\partial }{\partial  h_3}-h_7\frac{\partial }{\partial  h_4}-h_8\frac{\partial }{\partial  h_5} , 
 \label{V2}\\
 V_3 &=  h_4\frac{\partial }{\partial  h_1}+h_5\frac{\partial }{\partial  h_2}, 
 \label{V3}\\
 V_4 &= h_6\frac{\partial }{\partial  h_1}+h_7\frac{\partial }{\partial  h_2},
 \label{V4}\\
 V_5 &= h_7\frac{\partial }{\partial  h_1}+h_8\frac{\partial }{\partial  h_2}. 
 \label{V5}
\end{align}
Namely, equalities \eqref{fh1}--\eqref{fh5} have the form
$V_if=0$, $i = 1, \dots, 5$.

The vector fields $V_i$, $i=1,\dots, 5,$ form a Lie algebra with the product table
$\left[V_1, V_2 \right] = -V_3$, $\left[V_1, V_3 \right] = -V_4$,
$\left[V_2, V_3 \right] = -V_5$.
Denote for any $h \in L^*$ by $O_h$ the orbit of the fields $V_1, \dots, V_5$ passing through the point $h$~\cite{notes}. 
It is easy to see that $O_h$ is the orbit~\eq{co-orbit} of the co-adjoint action of the Lie group $G$ on $L^*$~\cite{marsden_ratiu, kirillov}.

By the Orbit Theorem \cite{notes}, $O_h$ is an immersed submanifold of $L^*$ of dimension
$$\dim O_h = \dim \Lie_h(V_1, \dots, V_5)= \dim \sspan (V_1(h), \dots, V_5(h)) = \rank J(h),$$
where
\begin{equation}
J(h)=(V_1, \dots, V_5)= \left(
\begin{array}{ccccc}
0 & h_3 & h_4 & h_6 & h_7\\
-h_3 & 0 & h_5 & h_7 & h_8\\
-h_4 & -h_5 & 0 & 0 & 0\\
-h_6 & -h_7 & 0 & 0 & 0\\
-h_7 & -h_8 & 0 & 0 & 0
\end{array} \right).
\end{equation}
Further, since
$O_h$ is a co-adjoint orbit, it is a symplectic, thus even-dimensional manifold, i.e., $\dim O_h \in \{0, 2, 4\}$.

Denote $\Delta = h_6h_8-h^2_7,$ and let $\Delta\neq 0$.  Since
\begin{equation}
\det
\left(
\begin{array}{cccc}
0 & h_3 & h_6 & h_7\\
-h_3 & 0 & h_7 & h_8\\
-h_6 & -h_7 & 0 & 0\\
-h_7 & -h_8 & 0 & 0
\end{array}\right)= -\Delta^2\neq 0, 
\end{equation}
then $\rank J(h) = \dim O_h =4.$ We have 
\begin{equation}
\label{ohsubset}
O_h  \subset \left\lbrace  h' \in L^* | C(h')= C(h), \,\, h_i (h') = h_i(h), \, \, i=6, 7, 8\right\rbrace . 
\end{equation}
The subset in the right-hand side of inclusion \eqref{ohsubset} is arcwise connected, thus this inclusion is in fact an equality.  In greater detail:
\begin{align}
O_h &= \R^2_{h'_1, h'_2} \times Q, \label{OQ}\\
Q &= \left\lbrace 
\left(h'_3, h'_4, h'_5\right) 
\in \R^3 | h'_3 = \left( h_6 \left(h'_5\right)^2 - 2h_7h'_4h'_5 + h_8\left( h'_4\right)^2 - C\right)/ (2\Delta)   \right\rbrace. \label{Q1}
\end{align}
If $\Delta>0$, then $Q$ is an elliptic paraboloid; and if $\Delta <0$,  then $Q$ is a hyperbolic paraboloid.  

So in the case $\D \neq 0$ the orbits $O_h$ are common level sets of the functions~\eq{Casimirs}. Any Casimir function is constant on the orbits $O_h$,  thus it depends functionally on the functions~\eq{Casimirs}.
\end{proof}

The next description of co-adjoint orbits follows from the previous proof.

\begin{corollary}
\label{cor:orbits}
let $h \in L^*$. Denote $\Delta = h_6h_8-h^2_7,$ $\Delta_1 = h_5h_7-h_4h_8,$ $\Delta_2 = h_5h_6-h_4h_7$. 
\begin{itemize}
\item[$(1)$]
The co-adjoint orbit  $\{ \Ad_{q^{-1}}^*(h) \mid q \in G \}$ coincides with the orbit $O_h$ of vector fields~\eq{V1}--\eq{V5} through the point $h$.
\item[$(2)$]
The orbits $O_h$ have the following dimensions:
\begin{itemize}
\item[$(2.1)$]
$\D^2 + \D_1^2 + \D_2^2 \neq 0 \quad \then \quad \dim O_h = 4$,
\item[$(2.2)$]
$\D^2 + \D_1^2 + \D_2^2 =0, \ h_3^2 + \dots + h_8^2 \neq 0 \quad \then \quad \dim O_h = 2$,
\item[$(2.3)$]
$h_3^2 + \dots + h_8^2 = 0 \quad \then \quad \dim O_h = 0$.
\end{itemize}
\item[$(3)$]
If $\D\neq 0$, then the orbit $O_h$ is described explicitly as~\eq{OQ}, \eq{Q1}.
\end{itemize}
\end{corollary}

In Subsec.~\ref{subsec:reduction} we consider the restriction of the vertical part of the Hamiltonian vector field $\vH$ to the orbit $O_h$, $\D \neq 0$.


\section{Pontryagin maximum principle}\label{sec:PMP}

In this section we apply a necessary optimality condition --- Pontryagin Maximum Principle
(PMP) \cite{PBGM,notes} to the sub-Riemannian problem~\eq{sys}--\eq{J} 
and derive ODEs for the geodesics of this problem. To this end
introduce the Hamiltonian of PMP 
\begin{align*}
&h_u^\nu(\lambda) = u_1 h_1(\lambda) + u_2 h_2(\lambda) +
\frac{\nu}{2}(u_1^2 + u_2^2), \\
&\lambda \in T^{*}G, \qquad u \in \mathbb{R}^2, \qquad \nu \in \mathbb{R}.
\end{align*}

\begin{theorem}[PMP, \cite{notes}]
Let $q(t)$, $t\! \in\! [0, t_1]$, be a SR minimizer corresponding to a control
$u(t)$,\; $t\!\in\![0, t_1]$. Then there exists a Lipschitzian curve $\lambda(t) \in T^{*}G,\; t \in [0, t_1]$,
$\pi(\lambda(t)) = q(t)$, and a number $\nu \in \{-1,\:0\}$ such that the following conditions hold:

\begin{enumerate}
 \item the Hamiltonian system of PMP
 \begin{equation} \label{Ham_gen}
 \dot{\lambda}(t) = \overrightarrow{h}^{\nu}_{u(t)} (\lambda(t)) \quad a.\,e.\; t \in [0,\:t_1], 
 \end{equation}
 \item the maximality condition
$h^{\nu}_{u(t)} (\lambda(t)) = \max \limits_{v \in \mathbb{R}^2} h^{\nu}_{v} (\lambda(t))$, $t \in [0,\:t_1]$,
   \item and the nontriviality condition
$  (\lambda(t), \nu) \ne (0,\,0)$, $t \in [0,\:t_1]$.
\end{enumerate}
\end{theorem}
In view of the product rule \eqref{X1X2}--\eqref{X1X5}, the Hamiltonian system~\eqref{Ham_gen}
reads in the parametrization $T^{*}G \ni \lambda = (h_1,\ldots,h_8,q)$ as follows:
\begin{align*}
&\dot{h}_1 = -u_2 h_3, \\
&\dot{h}_2 = u_1 h_3, \\
&\dot{h}_3 = u_1 h_4 + u_2 h_5, \\
&\dot{h}_4 = u_1 h_6 + u_2 h_7, \\
&\dot{h}_5 = u_1 h_7 + u_2 h_8, \\
&\dot{h}_6 = \dot{h}_7 = \dot{h}_8 = 0,\\
&\dot{q} = u_1 X_1 + u_2 X_2.
\end{align*}
In the next subsections we specialize the conditions
of PMP for the abnormal ($\nu = 0$) and normal
($\nu = -1$) cases.

\subsection{Abnormal case}
Let $\nu = 0$. Then the maximality condition
$h^{0}_{u} (\lambda) = u_1 h_1(\lambda) + u_2 h_2(\lambda) \rightarrow \max \limits_{u \in \mathbb{R}^2}$
yields  the identities along abnormal extremals:
$h_1(\lambda) = h_2(\lambda) = 0$.
Then
$0 = \dot{h}_1 = -u_2 h_3$ and  $0 = \dot{h}_2 = u_1 h_3$.
Since any minimizer can be reparametrized to have a constant velocity ($u_1^2 +u_2^2 \equiv \const$),
we have $u_1^2 +u_2^2 \ne 0$ along non--constant trajectory, thus abnormal extremals satisfy
one more identity:
$ h_3(\lambda) = 0$.
Then $0 = \dot{h}_3 = u_1 h_4 + u_2 h_5$, thus $(u_1(t),\,u_2(t)) = k(t)(-h_5(t),\, h_4(t))$
along abnormal extremals. After reparametrization of time we get the abnormal controls
$u_1 = -h_5$, $u_2 = h_4$.
Summing up, abnormal extremals $\lambda(t)$ are described as follows.

\begin{proposition}
Abnormal extremals  of the $(2,3,5,8)$ sub-Riemannian problem~\eq{sys}--\eq{J}
are reparameterizations of curves $\lambda(t) \in T^*G$ that satisfy the conditions
\begin{align}
&h_1(\lambda(t)) =  h_2(\lambda(t)) = h_3(\lambda(t)) = 0,\nonumber \\
&\begin{pmatrix} \dot{h}_4 \\ \dot{h}_5 \end{pmatrix} = D \begin{pmatrix} h_4 \\ h_5 \end{pmatrix}, \qquad
D = \begin{pmatrix} h_7 & -h_6 \\ h_8 & -h_7 \end{pmatrix}, \label{h4h5dot} \\
&\dot{h}_6 = \dot{h}_7 = \dot{h}_8 = 0, \nonumber \\
&\dot{q} = -h_5 X_1 + h_4 X_2.\nonumber
\end{align}
\end{proposition}

We have $\tr D = 0$, $\D = \det D = h_6h_8-h_7^2$, and the following cases are possible:
\begin{itemize}
\item[$(1)$]
$\D < 0$, then system~\eq{h4h5dot} has the saddle phase portrait, 
\item[$(2)$]
$\D > 0$, then system~\eq{h4h5dot} has the center phase portrait, 
\item[$(3)$]
$\D = 0$, $D \neq 0$, then    the  phase portrait of~\eq{h4h5dot} consists of lines and fixed points, 
\item[$(4)$]
$D = 0$, then    the  phase portrait of~\eq{h4h5dot} consists of   fixed points.
\end{itemize}
Thus follows that abnormal extremals are analytic (this is related to the famous open question on smoothness of sub-Riemannian minimizers~\cite{mont, monti}).

One can show that projections of abnormal extremal trajectories to the plane $\R^2_{x_1x_2}$ in these cases are respectively the following:
\begin{itemize}
\item[$(1)$]
hyperbolas, their separatrices, and center, 
\item[$(2)$]
homothetic ellipses and their center, 
\item[$(3)$]
parabolas, 
\item[$(4)$]
fixed points.
\end{itemize}
Trajectories that project to hyperbolas and parabolas are strictly abnormal (i.e., abnormal trajectories that are not normal trajectories~\cite{sussmann_liu, notes}). Moreover, one can parameterize the abnormal variety, i.e., the submanifold of $G$ filled by   abnormal trajectories~\cite{monti}. These results will appear in a forthcoming work~\cite{2358_abnorm}.

\subsection{Normal case}
Let $\nu = -1$. Then the maximality condition
$h^{-1}_{u} (\lambda) = u_1 h_1(\lambda) + u_2 h_2(\lambda) - \frac{1}{2}\left(u_1^2 + u_2^2\right) \rightarrow \max \limits_{u \in \mathbb{R}^2}$
yields the normal controls
$u_1 = h_1$,  $u_2 = h_2.$
Thus the normal extremals are trajectories of the Hamiltonian system
  \begin{equation}\label{Ham}
  \dot{\lambda} = \overrightarrow{H}(\lambda), \quad \lambda \in T^{*}G,
  \end{equation}
with the normal Hamiltonian
\be{H}
H = \frac{1}{2}\left(h_1^2 + h_2^2\right).
\ee
In the parametrization $T^{*}G \ni \lambda = (h_1,\ldots,h_8,q)$, system \eqref{Ham} reads as follows:
\begin{align}
&\dot{h}_1 = -h_2 h_3, \label{dh1}\\
&\dot{h}_2 = h_1 h_3, \label{dh2}\\
&\dot{h}_3 = h_1 h_4 + h_2 h_5, \label{dh3}\\
&\dot{h}_4 = h_1 h_6 + h_2 h_7, \label{dh4}\\
&\dot{h}_5 = h_1 h_7 + h_2 h_8, \label{dh5}\\
&\dot{h}_6 = \dot{h}_7 = \dot{h}_8 = 0, \label{dh8}\\
&\dot{q} = h_1 X_1 + h_2 X_2. \nonumber
\end{align}


\section{Integrability of the normal Hamiltonian field}\label{sec:normal}
In this section we study integrability of the  Hamiltonian field $\vH$. We compute 10 independent integrals of $\vH$, of which only 7 are in involution. Recall that for the Liouville integrability of the Hamiltonian system $\dlam = \vH(\lam)$ with 8 degrees of freedom we need 8 independent integrals in involution~\cite{arnold_mech}. After reduction by Casimir functions~\eq{Casimirs}, the vertical subsystem of $\vH$ shows numerically a chaotic dynamics, which leads to  Conjecture~\ref{conj:nonintegr} below on non-integrability of $\vH$. 

\subsection{Algebra of integrals of $\vH$}
The normal Hamiltonian system $\dlam = \vH(\lam)$ reads in the canonical coordinates $(\p_1, \dots, \p_8; x_1, \dots x_8)$ on $T^*G$ as follows:
\begin{align}
&\dot \p_1 = 
h_1 \left(x_1 \p_5 + \frac{x_2^2}{2} \p_7\right) 
- h_2\left(\frac 12 \p_3 + x_1 \p_4 + \frac{x_1^2}{2} \p_6 + \frac{x_1x_2}{2} \p_7\right), 
\nonumber\\
&\dot \p_2 =  
h_1\left(\frac 12 \p_3 + x_2 \p_5 + \frac{x_1 x_2}{2} \p_7 + \frac{x_2^2}{2} \p_8\right)-
h_2\left(x_2 \p_4 + \frac{x_1^2}{2} \p_7\right), 
\nonumber\\
&\dot \p_i = 0, \qquad i = 3, \dots, 8, 
\nonumber\\
&\dot q = h_1 X_1(q) + h_2 X_2(q), 
\nonumber\\
&h_1 = \p_1 - \frac{x_2}{2} \p_3 - \frac{x_1^2 + x_2^2}{2} \p_5 - \frac{x_1 x_2^2}{4} \p_7 - \frac{x_2^3}{6} \p_8, 
\label{h1}\\ 
&h_2 = \p_2 + \frac{x_1}{2} \p_3 + \frac{x_1^2 + x_2^2}{2} \p_4 + \frac{x_1^3}{6} \p_6 + \frac{x_1^2x_2}{4} \p_7.
\label{h2}
\end{align}

In view of results of Secs.~\ref{sec:realis}, \ref{sec:group}, the Hamiltonian field $\vH$ has the following integrals:
\begin{itemize}
\item
the system Hamiltonian $H$~\eq{H},
\item
right-invariant Hamiltonians $g_1$, \dots, $g_8$~\eq{g1}--\eq{gi},
\item
the Hamiltonian of rotation $h_0(\lam) = \langle \lam, X_0\rangle$~\eq{X0},
\item
Casimir functions  $h_6$, $h_7$, $h_8$, $C$~\eq{Casimirs},
\item
the cyclic variables $\p_3$, \dots, $\p_8$ of the Hamiltonian $H$~\eq{H}, \eq{h1}, \eq{h2}.  
\end{itemize}
Of these integrals, only 10 are functionally independent, thus we get an algebra of integrals
\be{I}
I = \spann(H, g_1, \dots, g_8, h_0)
\ee
with the nonzero brackets
\begin{gather}
\{h_0, g_4\} = g_5, \qquad \{h_0, g_5\} = -g_4, \label{h0g4}\\
\{h_0, g_6\} = 2 g_7, \qquad \{h_0, g_7\} = g_8-g_6, \qquad \{h_0, g_8\} = -2g_7.  \label{h0g6} 
\end{gather}
So we have an Abelaian algebra generated by 7 independent integrals:
\be{A}
A = \spann(H, g_3, \dots, g_8).
\ee
We proved the following statement.

\begin{theorem}
The normal Hamiltonian vector field $\vH$ has an algebra $I$~\eq{I}--\eq{h0g6}  of 10 independent integrals, and an Abelian algebra $A$~\eq{A} of 7 independent integrals.
\end{theorem}

Thus there lacks just one integral commuting with the integrals in $A$ in order to have Liouville integrability of $\vH$.

\subsection{Homogeneous integrals of $\vH$}
A natural source of integrals of $\vH$ are homogeneous polynomials in the momenta~$h_i$: \qquad
$
P_k = P_k(h_1, \dots, h_8)$, $\deg P_k = k$.
Although, for $k = 1, 2, 3$ we get no new integrals in this way, i.e., $P_1$, $P_2$, $P_3$ are expressed through the Casimir functions and the Hamiltonian $H$.

\begin{theorem}
Let a homogeneous polynomial $P_k(h_1, \dots, h_8)$ be an integral of the field $\vH$. Then:
\begin{itemize}
\item[$(1)$]
$P_1 = \sum_{i=6}^8 a_i h_i$, \qquad $a_i \in \R$,
\item[$(2)$]
$P_2 = \sum_{i, j=6}^8 a_{ij} h_i h_j + b H$, \qquad $a_{ij}, \ b \in \R$,
\item[$(3)$]
$P_3 = \sum_{i, j , l=6}^8 a_{ijl} h_i h_j h_l + H \sum_{i=6}^8 b_i h_i + a C$, \qquad $a_{ijl}, \ b_i, \ a \in \R$.
\end{itemize}
\end{theorem}
\begin{proof}
$(1)$
Let $P_1 = \sum_{i=1}^8 a_i h_i$, $a_i \in \R$, be an integral of $\vH$, then
\begin{align*}
0 &= \{H, P_1\} 
= - a_1 h_1h_3 + a_2 h_1h_3 + a_3(h_1h_4+h_2h_5) + a_4(h_1h_6+h_2h_7) \\
& \qquad\qquad\qquad +a_5(h_1h_7+h_2h_8),
\end{align*}
thus $a_1 = \dots = a_5 = 0$, so $P_1 = \sum_{i=6}^8 a_i h_i$.

Statements (2) and (3) are proved similarly.
\end{proof}

In addition to attempts to prove Liouville integrability of $\vH$, we tried also to apply noncommutative integrability theory~\cite{mish_fom}, but failed.

On the other hand, in the next subsection we present a numerical evidence of chaotic dynamics for the (reduction of) the Hamiltonian field $\vH$, which suggests thet this field is not Liouville integrable.

\subsection{Reduction of the vertical subsystem}\label{subsec:reduction}
The Hamiltonian field $\vH$ on $T^*G$ has a vertical part $\vH_{\textrm{vert}}$ defined on $L^*$ as follows (see e.g.~\cite{notes}):
$$
\vH_{\textrm{vert}}(\lam) = (\ad dH)^* \lam, \qquad \lam \in L^*.
$$
In the coordinates $(h_1, \dots, h_8)$ on $L^*$, the ODE $\dlam = \vH_{\textrm{vert}}(\lam)$ reads just as equations~\eq{dh1}--\eq{dh8}. 

For any $p = (h_6^0, h_7^0, h_8^0, C^0) \in \R^4$, consider the common level surface of the Casimir functions~\eq{Casimirs}
$$
O_p = \{\lam \in L^* \mid h_i(\lam) = h_i^0, \ i = 6, 7, 8, \ C(\lam) = C^0\}.
$$
By Corollary~\ref{cor:orbits}, in the generic case $\D^0 = h_6^0 h_8^0 - (h_7^0)^2 \neq 0$, the level set $O_p$ is an orbit of  co-adjoint action of the Lie group $G$ on $L^*$, it is 4-dimensional, and is parameterized by the coordinates $(h_1, h_2, h_4, h_5)$ as~\eq{OQ}, \eq{Q1}. In these coordinates, the restriction of the vertical subsystem $\dlam = \vH_{\textrm{vert}}(\lam)$ to $O_p$ reads as follows:
\begin{align*}
&\dot h_1 = - h_2 \ h_3(h_4,h_5), \\
&\dot h_2 = h_1 \ h_3(h_4,h_5), \\
&\dot h_4 = h_1 h_6^0 + h_2 h_7^0, \\
&\dot h_5 = h_1 h_7^0 + h_2 h_8^0, \\
&h_3(h_4,h_5) = (h_8^0 h_4^2 - 2 h_7^0 h_4h_5 + h_6^0h_5^2-C^0)/(2 \D^0).
\end{align*}
Restriction of this system to the level surface $\{H = 1/2\}$ gives, in the coordinates
\begin{gather*}
h_1 = \cos \th, \quad h_2 = \sin \th, \quad h_3 = c, \\ 
h_4 = a, \quad h_5 = b, \quad h_6 = m, \quad h_7 = p, \quad h_8 = n,
\end{gather*}
the following 3 equations:
\begin{align}
&\dot \th = (2 p ab - na^2 - mb^2)/(2 \Delta) + k, \label{red1}\\
&\dot a = m \cos \th + p \sin \th, \label{red2}\\
&\dot b = p \cos \th + n \sin \th, \qquad
m, \ n, \ p , \ k = \const. \label{red3}
\end{align}

If $\th(t)$ is increasing (or decreasing), then system~\eq{red1}--\eq{red3} defines a Poincar\`e mapping
\begin{align*}
&\map{P}{\R^2}{\R^2}, \qquad P(a,b) = (a', b'), \\
&\restr{(\th(t), a(t), b(t))}{t=0} = (0, a, b), \\
&\restr{(\th(t), a(t), b(t))}{t=T>0} = (2\pi, a', b').
\end{align*}
We computed numerically the orbits 
$
\{P^i(a,b) \mid i \in \N\}$,
and for various values of the parameters $(m,n,p,k)$ and initial points $(a,b)$, we get regular or chaotic bahaviour, see Figs.~\ref{fig:reg1}--\ref{fig:chaos6}.
This numeric evidence leads to the following

\begin{conjecture}\label{conj:nonintegr}
\begin{itemize}
\item[$(1)$]
The Hamiltonian vector field $\vH$ is not Liouville integrable on $T^*G$.
\item[$(2)$]
There exist symplectic submanifolds $S \subset T^*G$, $0 < \dim S < \dim  T^*G$, such that $\vH$ is Liouville integrable on $S$.
\end{itemize}
\end{conjecture}

\onefiglabelsizen
{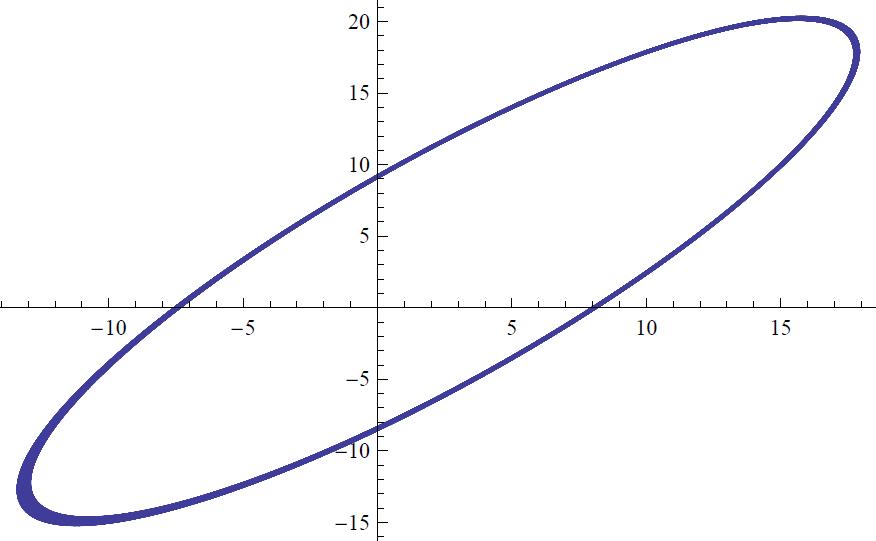}
{Regular orbit of Poincar\'e map ($5 \cdot 10^5$ points)}  
{fig:reg1}
{5}

\onefiglabelsizen
{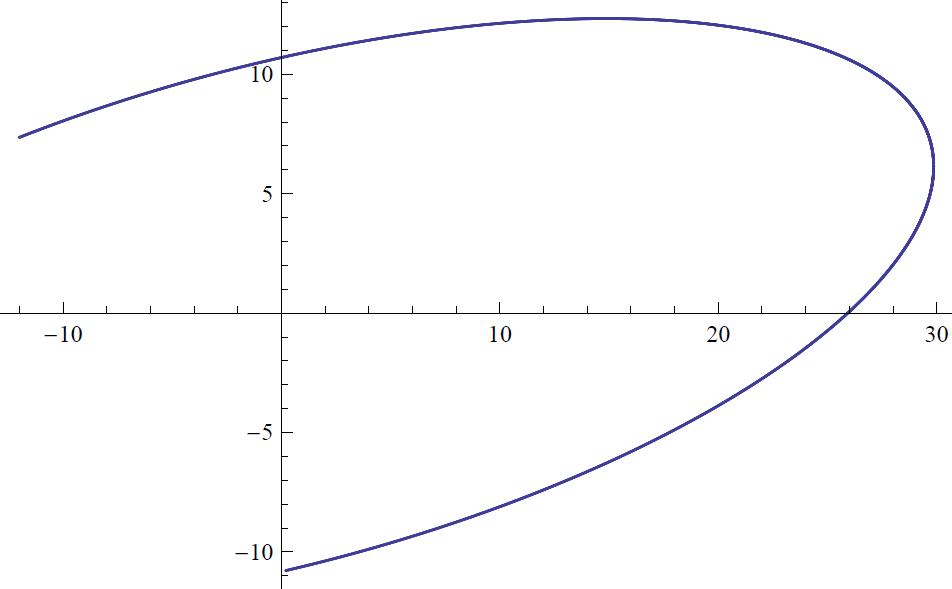}
{Regular orbit of Poincar\'e map ($5 \cdot 10^5$ points)}  
{fig:reg2}
{5}

\onefiglabelsizen
{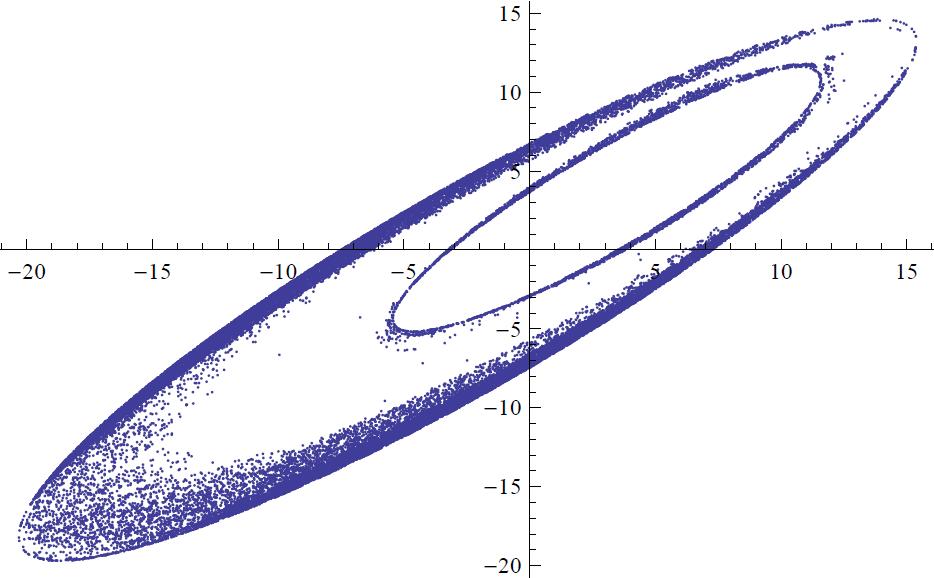}
{Chaotic orbit of Poincar\'e map ($ 5 \cdot 10^6$ points)}  
{fig:chaos1}
{5}

\onefiglabelsizen
{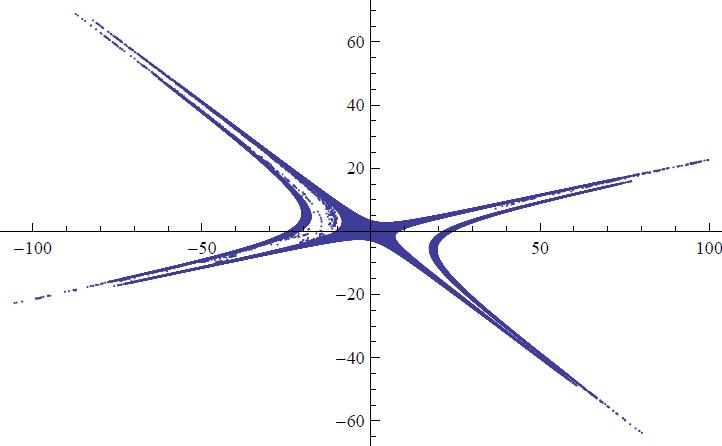}
{Chaotic orbit of Poincar\'e map ($ 5 \cdot 10^5$ points)}  
{fig:chaos5}
{5}

\onefiglabelsizen
{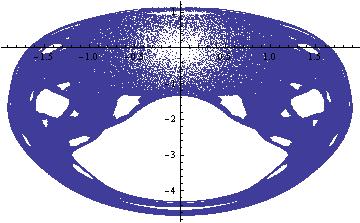}
{Chaotic orbit of Poincar\'e map ($ 5 \cdot 10^5$ points)}  
{fig:chaos2}
{5}

\onefiglabelsizen
{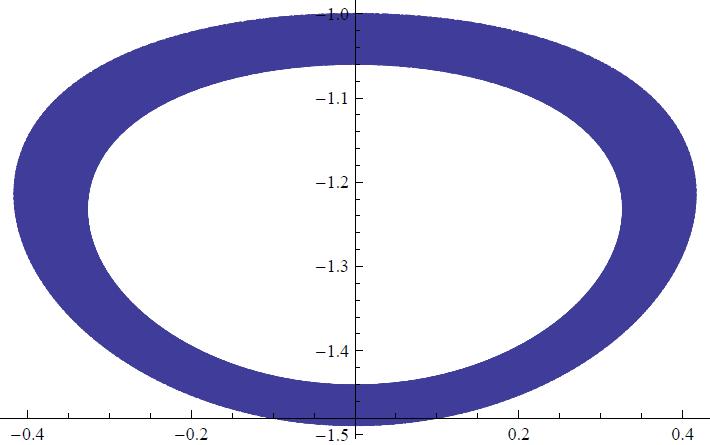}
{Chaotic orbit of Poincar\'e map ($ 5 \cdot 10^5$ points)}  
{fig:chaos6}
{5}

\subsection{Lower-dimensional projections}
For special initial values of $\lam \in L^*$, projections of normal geodesics $q(t)$ of the $(2,3,5,8)$-problem to certain subspaces of the state space $\R^8$ yield geodesics of lower-dimensional sub-Riemannian problems since there is an obvious nested chain of nilpotent SR problems on Carnot groups, like Russian Matryoshka:
$$
(2) \subset (2,3) \subset (2,3,5 ) \subset (2,3,5,8),
$$
corresponding to the chain of subspaces:
$$
\R^2_{x_1x_2} \subset \R^3_{x_1x_2x_3} \subset \R^5_{x_1\dots x_5} \subset \R^8_{x_1\dots x_8}. 
$$
Multiplication table in the Heisenberg algebra (growth vector $(2,3)$) is
\be{23}
[X_1, X_2] = X_3, 
\ee
and in the Cartan algebra (growth vector $(2,3,5)$) is
\be{235}
[X_1, X_2] = X_3, \quad [X_1, X_3] = X_4, \quad [X_2, X_3] = X_5.
\ee
Multiplication tables~\eq{23} and~\eq{235}
are depicted resp. in Figs.~\ref{fig:23} and~\ref{fig:235} (compare with Fig.~\ref{fig:2358} for the (2,3,5,8) Carnot algebra).

\begin{figure}[htb]
\setlength{\unitlength}{1cm}

\begin{center}
\begin{picture}(4, 4)
\put(1.15, 1.9){ \vector(1, -1){0.8}}
\put(2.85, 1.9){ \vector(-1, -1){0.8}}


\put(1, 1.98) {$X_1$}
\put(3, 1.98) {$X_2$}
\put(2, 0.75) {$X_3$}

\end{picture}

\caption{The Heisenberg algebra \label{fig:23}}
\end{center}
\end{figure}
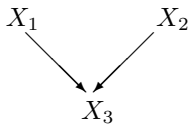

\begin{figure}[htb]
\setlength{\unitlength}{1cm}

\begin{center}
\begin{picture}(4, 4)
\put(1.15, 2.9){ \vector(1, -1){0.8}}
\put(1, 2.9){ \vector(0, -1){2}}
\put(2.85, 2.9){ \vector(-1, -1){0.8}}
\put(3, 2.9){ \vector(0, -1){2}}
\put(1.9, 1.65){ \vector(-1, -1){0.8}}
\put(2.1, 1.65){ \vector(1, -1){0.8}}


\put(1, 0.5) {$X_4$}
\put(3, 0.5) {$X_5$}
\put(1, 2.98) {$X_1$}
\put(3, 2.98) {$X_2$}
\put(2, 1.75) {$X_3$}

\end{picture}

\caption{The Cartan algebra \label{fig:235}}
\end{center}
\end{figure}
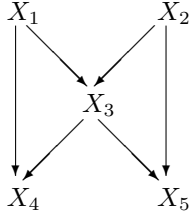

If $h_3(\lam) = \dots = h_8(\lam) = 0$, then $(x_1(t), x_2(t))$ is a Riemannian geodesic in the Euclidean plane $\R^2_{x_1x_2}$, i.e., a straight line.

If $h_4(\lam) = \dots = h_8(\lam) = 0$, then $(x_1(t), x_2(t), x_3(t))$ is a sub-Riemannian geodesic in the Heisenberg group $\R^3_{x_1x_2x_3}$, thus the curve $(x_1(t), x_2(t))$ is a straight line or a circle~\cite{versh_gersh, brock}.

If $h_6(\lam) = h_7(\lam)  = h_8(\lam) = 0$, then $(x_1(t), \dots, x_5(t))$ is a sub-Riemannian geodesic in the Carnot group $\R^5_{x_1\dots x_5}$, thus the curve $(x_1(t), x_2(t))$ is an Euler elastica --- a  stationary configuration of elastic rod in the plane~\cite{euler, love, dido_exp, max1, max2, max3, el_max, el_conj, el_exp}, see the plots of elasticae for various values of elastic energy at Figs.~\ref{fig:el1}--\ref{fig:el4}.

\twofiglabelsize
{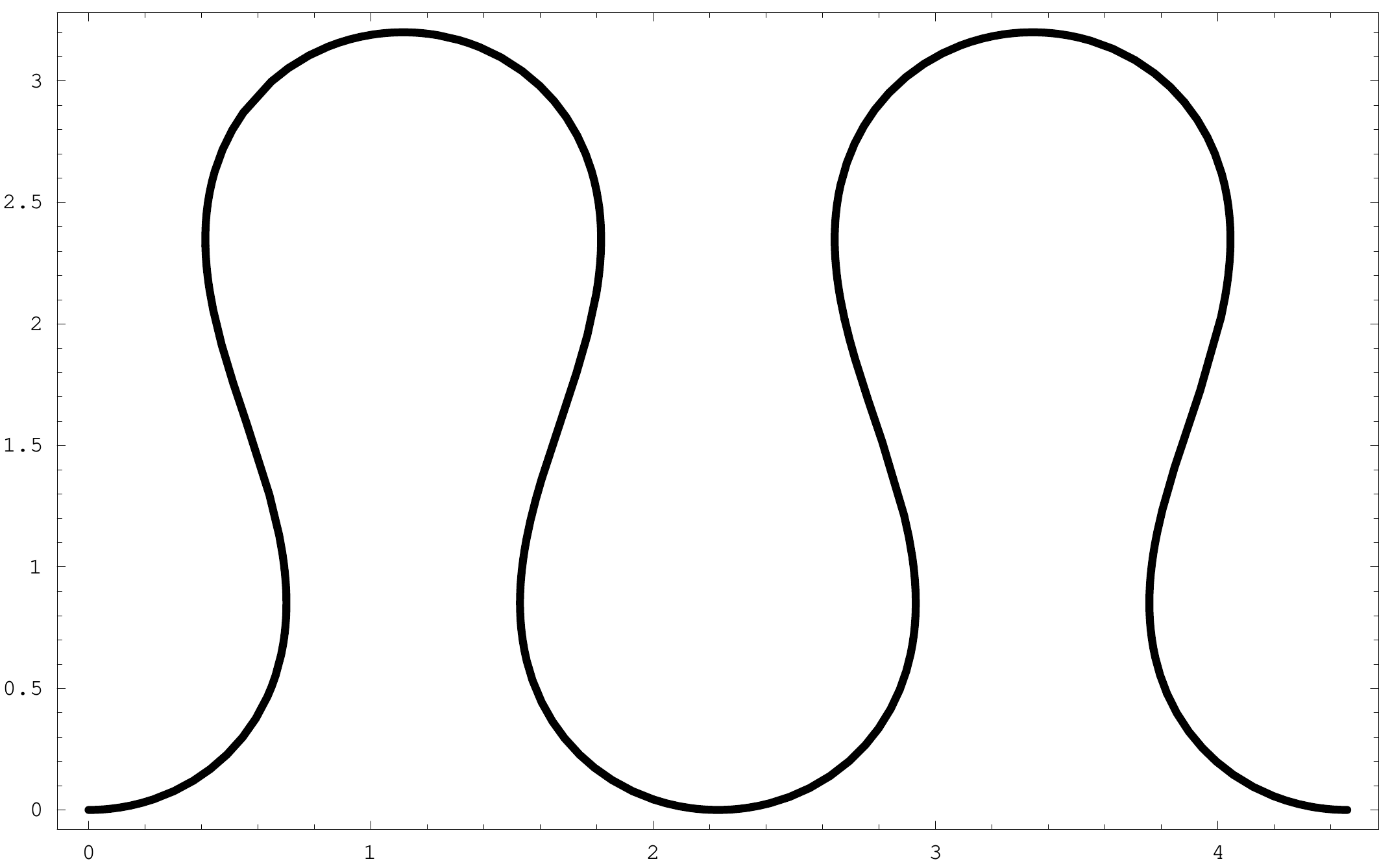}{Inflexional elastica}{fig:el1}
{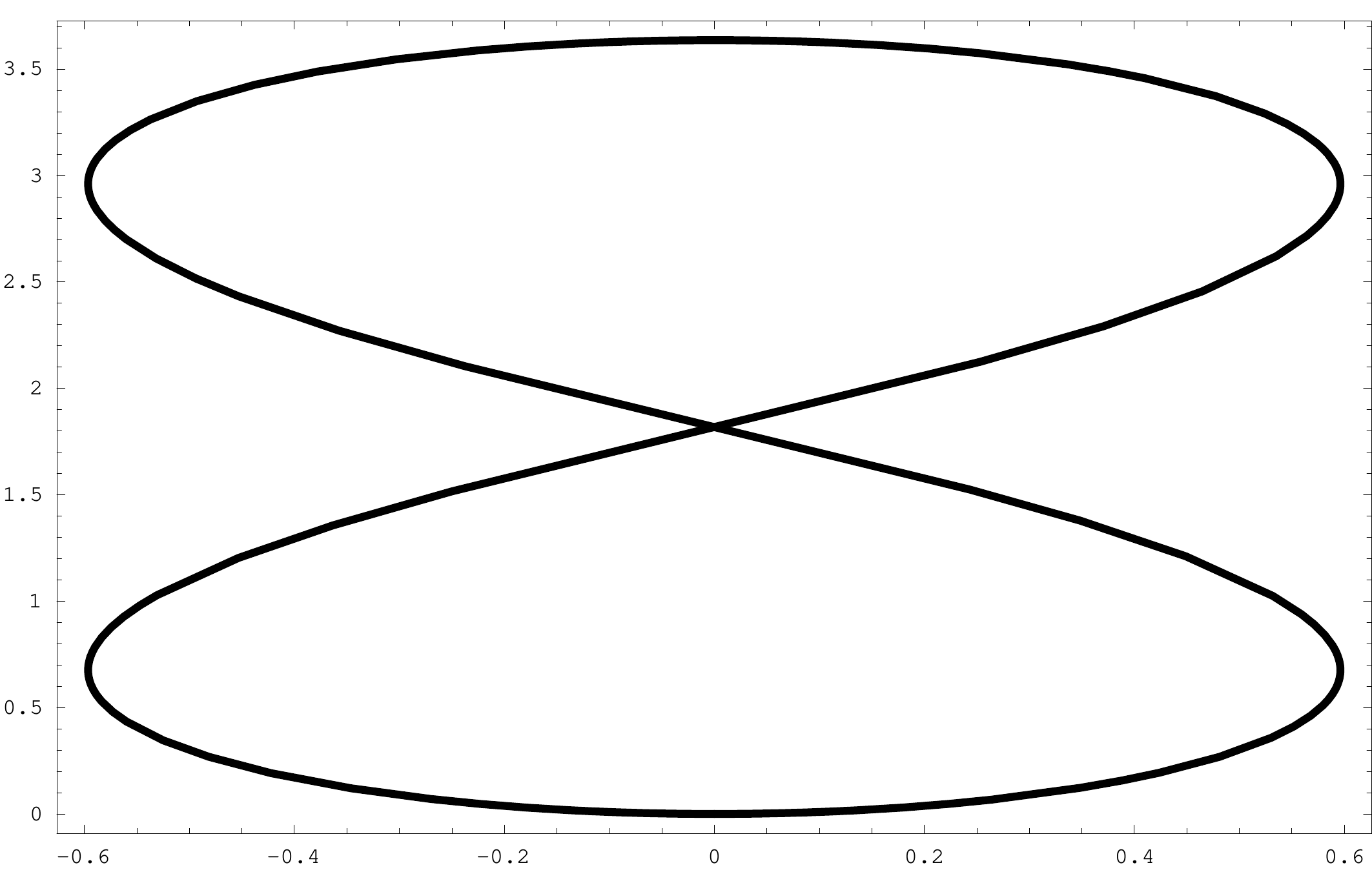}{Inflexional elastica}{fig:el2}
{0.4}{0.4}

\twofiglabelsize
{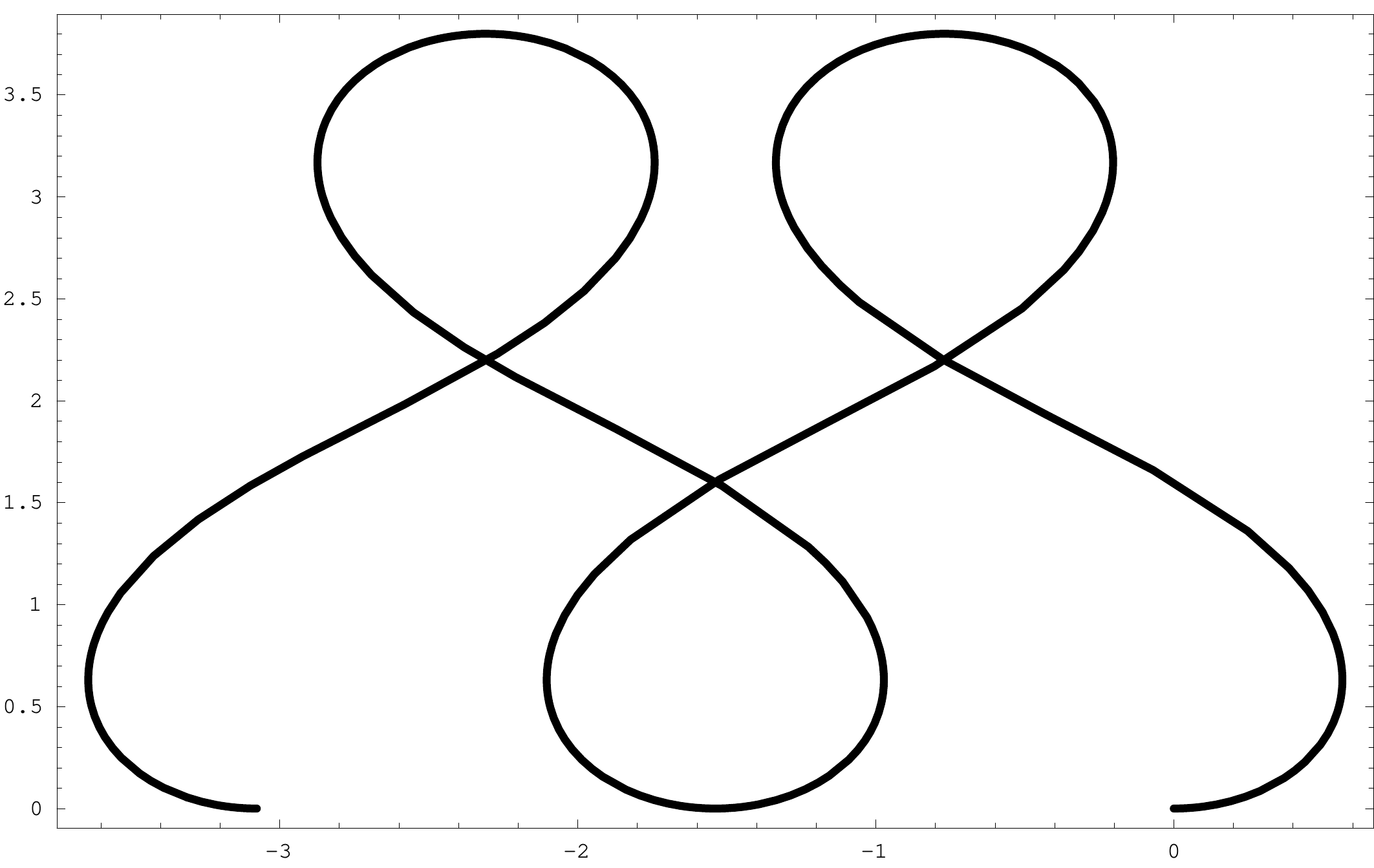}{Inflexional elastica}{fig:el3}
{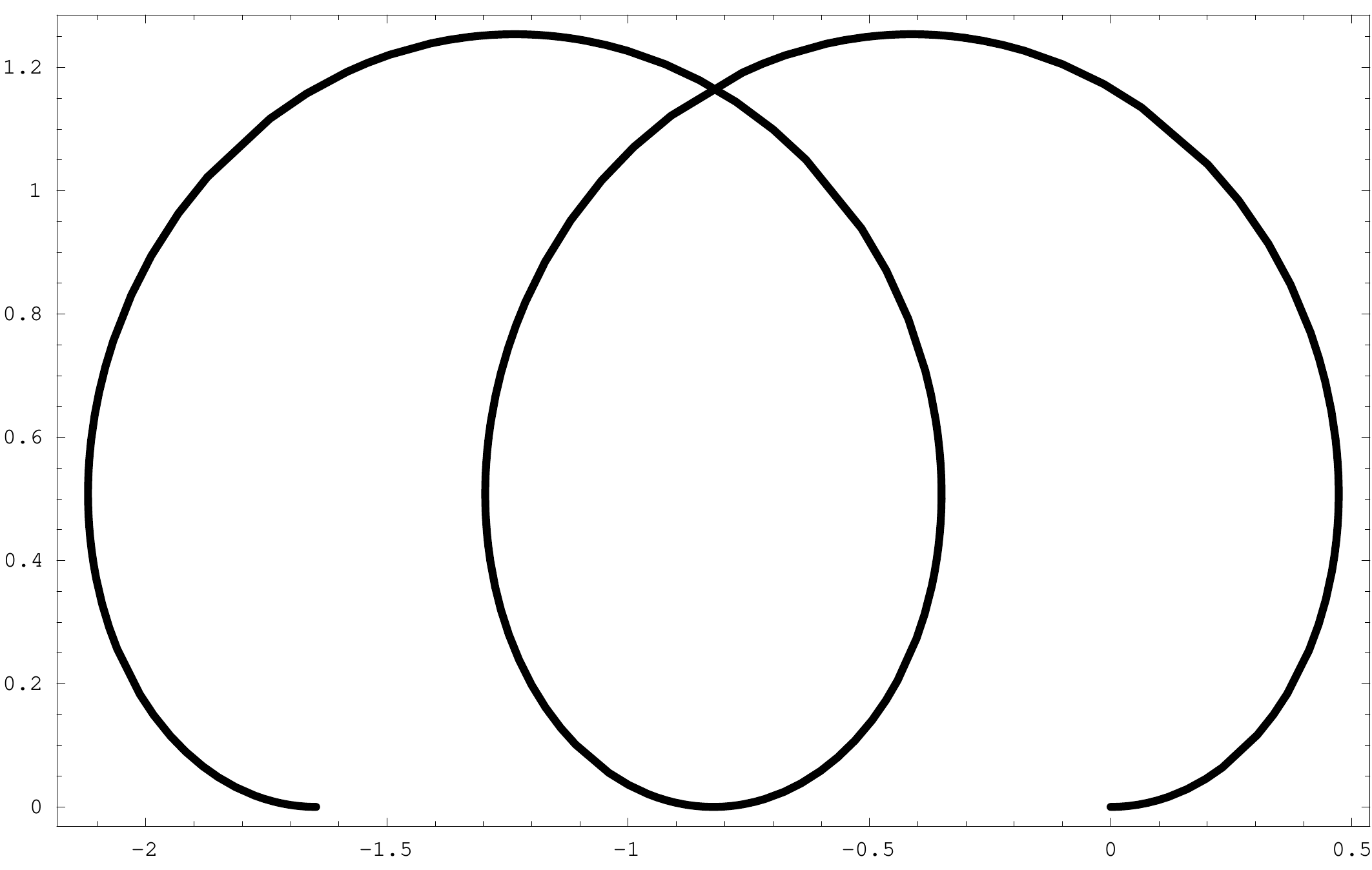}{Non-inflexional elastica}{fig:el4}
{0.4}{0.4}

For generic $\lam \in L^*$, the curves $(x_1(t), x_2(t))$ look like ``elasticae of variable elastic energy'', see Figs.~\ref{fig:gen_el1}, \ref{fig:gen_el2}.

\onefiglabelsizen
{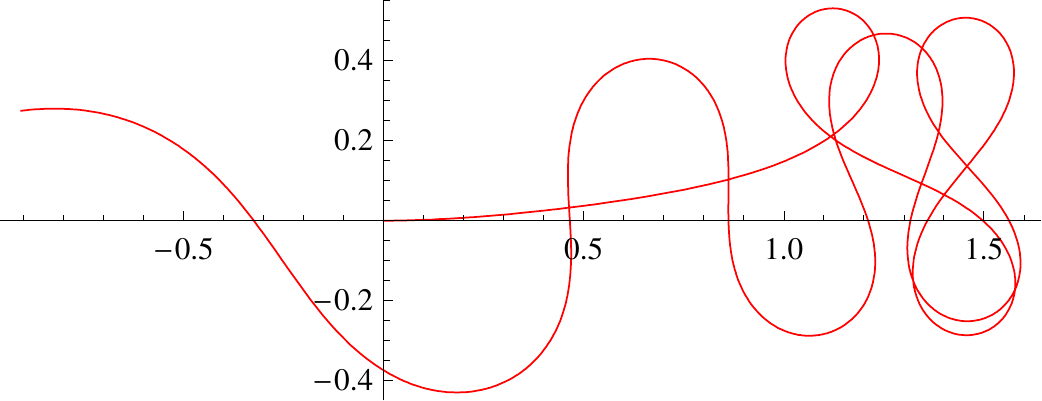}
{Elastica of variable elastic energy}  
{fig:gen_el1}
{4.5}

\onefiglabelsizen
{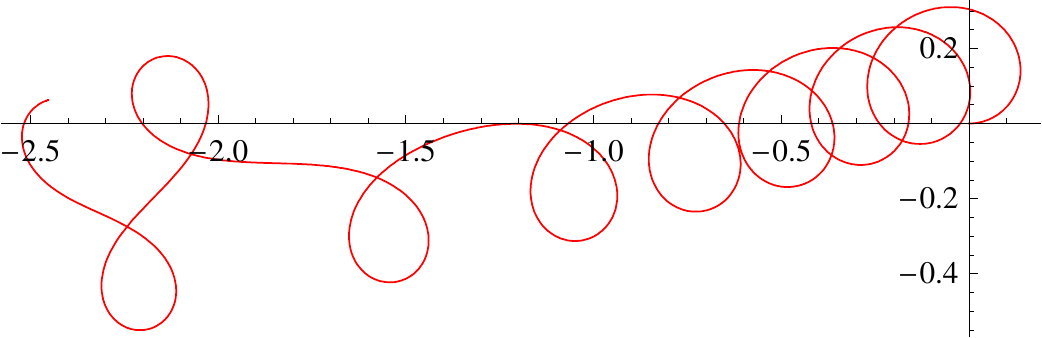}
{Elastica of variable elastic energy}  
{fig:gen_el2}
{4.}

There is an obvious relation of optimality of trajectories of the (2,3,5,8)-problem and its lower-dimensional projections due to the following simple statement.

\begin{proposition}[\cite{engel_conj}]
\label{prop:project}
Consider two optimal control problems:
\begin{eqnarray*}
&& \dot{q}^i = f^i(q^i, u),\quad q^i \in M^i, \quad u \in U, \\
&&q^i(0)= q^i_0, \quad q^i(t_1)= q^i_1,\\
&&J= \int^{t_1}_0 \varphi(u) \,dt \rightarrow \min,\\
&& i=1,2.
\end{eqnarray*}
Suppose that there exists a smooth map $G:M^1 \rightarrow M^2$, s.~t. if $q^1(t)$ is the trajectory of the first system corresponding to a control $u(t)$, then $q^2(t)= G(q^1(t))$ is the trajectory of the second system with the same control.

Further assume that $q^1(t)$ and $q^2(t)$ are such trajectories. If $q^2(t)$ is locally (globally) optimal for the second problem, then $q^1(t)$ is locally (globally) optimal for the first problem.
\end{proposition}

This proposition provides lower bounds for the cut time
$$
\tcut(\lambda) = \sup \{ t>0 \mid \pi \circ e^{s \vH} (\lambda) \text{ is globally optimal for } s \in [0,t]\}
$$
and the first conjugate time
$$
\tconj(\lambda)=\sup \left\{t>0 \mid \pi \circ e^{s \vH} (\lambda) \text{ is locally optimal for } s \in [0,t] \right\}
$$
of the (2,3,5,8)-problem in terms of the same functions for its lower-dimensional projections.

For the Riemannian problem on the plane, the straight lines are optimal forever, so the cut and first conjugate times are $+ \infty$, thus for the (2,3,5,8)-problem
$$
h_3(\lam) = \dots h_8(\lam) = 0 \quad \then \quad 
\tcut(\lam) = \tconj(\lam) = + \infty.
$$
For the sub-Riemannian problem on the Heisenberg group, the circles are locally and globally optimal up to the first loop, thus for the (2,3,5,8)-problem
$$
h_3(\lam) \neq 0,  \quad h_4(\lam) =  \dots = h_8(\lam) = 0 \quad \then \quad 
\tconj(\lam) \geq \tcut(\lam) \geq \frac{2\pi}{|h_3(\lam)|}.
$$
Similar, but much more complicated bounds hold for the case $h_6(\lam) =  h_7(\lam) = h_8(\lam) = 0$ via comparison with the cut and first conjugate times for the sub-Riemannian problem on the Cartan group~\cite{dido_exp, max1, max2, max3}.

\section{Conclusion}\label{sec:concl}
We see the following interesting questions for the (2,3,5,8)-problem:
\begin{enumerate}
\item
study optimality of abnormal geodesics,
\item
describe all cases where the normal Hamiltonian vector field $\vH$ is Liouville intergable, integrate and study the corresponding normal geodesics,
\item
describe precisely the chaotic dynamics of the normal Hamiltonian vector field $\vH$.
\end{enumerate}

We plan to address these questions in forthcoming works.

\end{document}